\documentclass[twoside,a4paper,reqno,11pt]{amsart} 
\usepackage[top=28mm,right=30mm,bottom=28mm,left=30mm]{geometry}

\usepackage{amsfonts, amsmath, amssymb, mathrsfs, bm, latexsym, stmaryrd, array, hyperref, mathtools, bold-extra}

\usepackage{etaremune}
\usepackage{enumerate}

\newcommand{\leqs}{\leqslant}
\newcommand{\geqs}{\geqslant}

\newcommand{\soc}{\operatorname{soc}}
\newcommand{\Aut}{\operatorname{Aut}}

\newcommand{\vs}{\vspace{3mm}}

\newcommand{\Syl}{{\mathrm {Syl}}}

\newcommand{\LL}{\mathrm {L}}

\newcommand{\PGU}{\mathrm {PGU}}
\newcommand{\PSL}{\mathrm {L}}
\newcommand{\SL}{\mathrm {SL}}

\newcommand{\GL}{\mathrm {GL}}
\newcommand{\PGL}{\mathrm {PGL}}
\newcommand{\Irr}{\mathrm {Irr}}
\newcommand{\Alt}{\mathrm {A}}
\newcommand{\UU}{\mathrm {U}}
\newcommand{\PP}{\mathcal {P}}
\makeatletter
\newcommand{\imod}[1]{\allowbreak\mkern4mu({\operator@font mod}\,\,#1)}
\makeatother

\renewcommand{\mod}[1]{\imod{#1}}
\renewcommand{\leq}{\leqs}
\renewcommand{\geq}{\geqs}

\newtheorem{theorem}{Theorem} 
\newtheorem{conjecture}{Conjecture}

\newtheorem{corol}[theorem]{Corollary}

\newtheorem{thm}{Theorem}[section] 
\newtheorem{prop}[thm]{Proposition} 
\newtheorem{lem}[thm]{Lemma}
\newtheorem{cor}[thm]{Corollary}

\theoremstyle{definition}
\newtheorem{rem}[thm]{Remark}

\newtheorem{ex}[thm]{Example}
\newtheorem*{problem}{Problem}
\begin{document}
\title[Weakly subnormal subgroups and the Baer-Suzuki theorem]{Weakly subnormal subgroups and variations of the Baer-Suzuki theorem}

\author{Robert M. Guralnick}
\address{R.M. Guralnick, Department of Mathematics, University of Southern California, Los Angeles, CA 90089-2532, USA}
\email{guralnic@usc.edu}

\author{Hung P. Tong-Viet}
\address{H.P. Tong-Viet, Department of Mathematics and Statistics, Binghamton University, Binghamton, NY 13902-6000, USA}
\email{htongvie@binghamton.edu}

\author[G. Tracey]{Gareth Tracey}
\address{G. Tracey, Mathematics Institute, University of Warwick, Coventry, CV4 7AL, UK}
\email{gareth.tracey@warwick.ac.uk}

\renewcommand{\shortauthors}{Guralnick, Tong-Viet, Tracey}

\begin{abstract}
A subgroup $R$ of a finite group $G$ is weakly subnormal in $G$ if $R$ is  not subnormal in $G$ but it is subnormal in every proper overgroup of $R$ in $G$. In this paper, we first classify all finite groups $G$  which contains a weakly subnormal $p$-subgroup for some prime $p$. We then determine all  finite groups containing a cyclic weakly subnormal  $p$-subgroup. As applications, we prove a number of variations of the Baer-Suzuki theorem using the orders of certain group elements. 
\end{abstract}

\date{\today}

\maketitle

\setcounter{tocdepth}{1}
\tableofcontents

\section{Introduction}\label{s:intro}

Let $G$ be a finite group and let $p$ be a prime divisor of the order of $G$. A subgroup $R$ of $G$ is weakly subnormal in $G$ if $R$ is not subnormal in $G$ but $R$ is subnormal in every proper overgroup of $R$ in $G$. The first main goal of this paper is to determine the structure of all finite groups $G$ containing a weakly subnormal $p$-subgroup $R$. 
Note that if $R$ is a $p$-group, then $R$ is weakly subnormal in $G$ if and only if $RO_p(G)$ is weakly subnormal in $G$ if and only if $RO_p(G)/O_p(G)$ is weakly subnormal in $G/O_p(G)$.  So we will generally
assume that $O_p(G)=1$.  Wielandt's Zipper Lemma  implies that if $R$ is weakly subnormal in $G$, then $R$ is contained in a unique maximal subgroup $M$ and if $R$ is a $p$-group,  then $R \le O_p(M)$.  Moreover,
$M$ must be self-normalizing or $O_p(M)$ would be normal in $G$.  

We will essentially classify all possibilities of weakly subnormal $p$-subgroups of finite groups, showing that there are very significant restrictions on them.   Our results depend on  recent papers  \cite{BBGH} and \cite{GT} considering when a Sylow subgroup is contained in a unique maximal subgroup or a cyclic subgroup
is contained in a unique maximal subgroup.   

Before stating our main theorems, we fix some standard notation. For an element $g$ of a group $G$, we will write $o(g)$ for the order of $g$. We will write $\Phi(G)$, $F(G)$, $E(G)$, and $F^*(G)$ for the Frattini subgroup, Fitting subgroup, layer, and generalized Fitting subgroup of $G$, respectively. 

Our first theorem classifies the easy case of weakly subnormal $p$-subgroups: the case where $G$ is $p$-solvable. Recall that a $p$-group $P$ is \emph{special} if it is either elementary abelian, or satisfies $\Phi(P)=[P,P]=Z(P)$.
\begin{theorem} \label{t:psolvable}   Let $p$ be a prime and let $G$ be a finite $p$-solvable group with $O_p(G)=1$.  If $R$ is a weakly subnormal $p$-subgroup of $G$,  
then $G=QR$ where $Q$ is a special normal $q$-subgroup of $G$ for some prime $q \ne p$,  $R$ centralizes $\Phi(Q)$, and $R$ acts irreducibly on  $Q/\Phi(Q)$. In particular, $G$ is solvable
and $R$ is a Sylow $p$-subgroup of $G$.
\end{theorem}

The analysis of the case when $G$ is not $p$-solvable is more intricate.  Here is the main theorem.   

 \begin{theorem} \label{t:main1}  Let $p$ be a prime, let $G$ be a finite group  with $O_p(G)=1$, and assume that $G$ is not $p$-solvable.
Let $R$ be a weakly subnormal $p$-subgroup of $G$.    Then either $E:=F^*(G)$ is quasisimple; or $p=2$ and either $E$ is a minimal
normal subgroup \emph{(}and so $Z(E)=\Phi(G)=1$\emph{)}, or $E(G)$ has a center of order $3$ and is a central product of copies of $3\cdot A_6$.  Moreover, one of the following holds:

\begin{itemize}  
\item[{\rm (i)}]     $G$ is  quasisimple and $G/Z(G)$ is recorded in  \cite[Table E]{BBGH}. 
\item[{\rm (ii)}] $p=5$ and $G={^2}{\rm B}_2(32).5$.
\item [{\rm(iii)}]  $p=3$,  and $G$ is one of  $\LL_2(8).3$ or $\UU_3(8) < G \le {\PGU}_3(8).3$; or 
\item[{\rm (iv)}] $p=2$ and  $G= \PGL_2(q)$ with $q \ge 7$ a Fermat or Mersenne prime or $q=9$.
\item[{\rm(v)}]  $p=2$ and $G= \LL_3(4).2_3$, ${\rm M}_{10}$ or $\mathrm{Aut}(\Alt_6)$.
\item[{\rm(vi}]   $p=2$ and $G = \LL_2(q).2^2$ or $\LL_2(q).2_3$ with $q = 81$ or $q=r^2$ with $r \ge 5$ a Fermat prime \emph{(}the nonsplit extension\emph{)}.
\item[{\rm(vii)}]  $p=2$ and  $G=\LL_2(q)$ or $\PGL_2(q)$ with $q$ a prime and $q \equiv -1 \pmod 8$ and $|R| \ge 8$.
\item[{\rm(viii)}] $p=2$ and   $G=\LL_3(3).2$.
\item [{\rm(ix)}]  $p=2$,  $G=E(G)R$ and $E(G)  = T_1 \times \ldots \times T_t, t > 1$ is a minimal normal subgroup 
and if $T = T_1$, then $N_G(T)/C_G(T)$ has a maximal Sylow $2$-subgroup and  $N_G(T)/C_G(T)$ is isomorphic to one of
\[
{\rm PGL}_2(7), \; {\rm PGL}_2(9), \; {\rm M}_{10}, \; {\rm L}_{2}(9).2^2, \; {\rm L}_2(q),\; {\rm PGL}_2(q),
\]
where $q >7$ is a Fermat or Mersenne prime.
\item[{\rm(x)}]  $p=2$, $G=E(G)R$ and $E(G)$ is a central product of triple covers of $\Alt_6 = \LL_2(9)$,  $E(G)$ has a center
of order $3$ and if $T$ is a component of $G$, then $N_G(T)/C_G(T) = {\rm {M}}_{10}$.  
\end{itemize} 
\end{theorem}

\begin{rem} 
\ From the previous theorems, one obtains the classification of maximal weakly subnormal $p$-subgroups.  If $R$ is a weakly subnormal
$p$-subgroup and $M$ is the unique maximal subgroup containing $R$, then $M$ is the only maximal subgroup containing $O_p(M)$
and $R \le O_p(M)$. Thus, $O_p(M)$ is the unique  (up to conjugacy) maximal weakly subnormal $p$-subgroup of $G$.    In all cases with
$p \ne 2$,  $O_p(M)$ is a Sylow $p$-subgroup of $G$ (and this is true in many but not all cases with $p=2$ as well).  
\end{rem}

For applications, we need the classification of cyclic weakly subnormal $p$-subgroups.   As usual, we assume that $O_p(G)=1$.   If $G$ is $p$-solvable, then the classification is given
in Theorem \ref{t:psolvable} (in that case  Sylow $p$-subgroups are the only such examples), and so we assume this is not the case.  

Suppose that $R$ is a cyclic weakly subnormal $p$-subgroup of such a group $G$. Then $R$ is contained in a unique maximal subgroup $M$ and moreover, $R \le O_p(M)$.  If $P$ is a Sylow $p$-subgroup
containing $R$, then $M$ is the only maximal subgroup containing $P$.  Conversely, if $R \le O_p(M)$ is cyclic and $M$ is the unique maximal subgroup containing $R$ (and $M$ is not normal
in $G$), then $R$ is weakly
subnormal.  Thus, one just has to check the cases in Theorem \ref{t:main1}, consider elements $x$ in $O_p(M)$, and check to see that $x$ is contained
in no other maximal subgroups.  Another approach is to use the results of \cite{GT} where there is a classification of cyclic subgroups contained in a unique maximal
subgroup and check to see if they are contained in $O_p(M)$.

Note that for $p$ odd, in all but two cases (one each for $p=3$ or $5$),  the group is quasisimple; the Sylow $p$-subgroup is cyclic; and the  unique maximal subgroup
is the normalizer of the Sylow $p$-subgroup.   Moreover, if $G$ happens to be a quasisimple finite group of Lie type with $p$ odd, the elements are either regular semisimple
or unipotent.  The only cases where $R$ is unipotent is if $G= \LL_2(p)$ or $SL_2(p)$ with $p \ge 5$.   If $R$ consists of semisimple elements, then it cannot be contained
in a proper parabolic subgroup since then it would be conjugate to a subgroup of a Levi subgroup and so would be contained in at least two parabolic subgroups.   In
particular, $R$ is generated by a regular semisimple element.

\begin{theorem} \label{t:cyclic}    Let $G$ be a finite group and $p$ be a prime with $O_p(G)=1$.  Assume that $G$ is not $p$-solvable and $R$ is a cyclic weakly subnormal $p$-subgroup.
Let $M$ denote the unique maximal subgroup of $G$ containing $R$. 
Then one of the following holds:
\begin{enumerate}[{\rm(i)}]
    \item  $G$ is quasisimple, a Sylow $p$-subgroup of $G$ is cyclic, $R$ is any nontrivial $p$-subgroup, $M = N_G(R)$, and $(G/Z(G),M/Z(G))$ is given in Table \ref{tab:thm3}.
\item   $p=5$,  $G={^2}{\rm B}_2(32).5$,  $R$ is any cyclic subgroup of order $25$ not contained in the socle, and $M$ is the normalizer of a nonsplit torus of order $25$.
\item $p=3$, $G=\LL_2(8).3$,  $R$ is any cyclic subgroup of order $9$ not contained in the socle, and $M$ is the normalizer of a nonsplit torus.
\item $p=2$, $G={\rm M}_{10}$, $R$ is any group of order $8$ not contained in the socle, and $M$ is a Sylow $2$-subgroup.
\item $p=2$,  $G=\LL_2(q)$ or $\PGL_2(q)$, $M$ is the normalizer of a nonsplit torus, $q$ is prime, $q \equiv -1 \pmod 8$, and  $|R| \ge 8$.
\item  $p=2$,  $G=E(G)R$ and $E(G)  = T_1 \times \ldots \times T_t, t > 1$ is a minimal normal subgroup 
and if $T = T_1$, then $N_G(T)/C_G(T)$ has a maximal Sylow $2$-subgroup and  $N_G(T)/C_G(T)$ is isomorphic to one of
\[
{\rm PGL}_2(7), \; {\rm M}_{10},  \; {\rm L}_2(q),\; {\rm PGL}_2(q),
\]
where $q >7$ is a Mersenne prime.
\item  $p=2$, $G=E(G)R$ and $E(G)$ is a central product of triple covers of $\Alt_6 = \LL_2(9)$,  $E(G)$ has a center
of order $3$ and if $T$ is a component of $G$, then $N_G(T)/C_G(T) = {\rm M}_{10}$.  
\end{enumerate} 
\end{theorem}

\begin{corol}\label{c:cyclic34}
Let $G$ be a finite group, and assume that $G$ has a non-trivial weakly subnormal cyclic subgroup $R$.
\begin{enumerate}
    \item[\rm(i)] If $|R|=2$, then $G$ is dihedral of order $2q$, for an odd prime $q$.
    \item[\rm(ii)] If $|R|=3$, then $G$ is either solvable or $G/O_3(G)\cong \LL_2(2^e)$ with $e$ an odd prime.
    \item[\rm(iii)] If $|R|=4$, then $G$ is solvable.
\end{enumerate}
\end{corol}

\begin{rem}
In Table \ref{tab:thm3}, we adopt similar notation to that used in \cite{BBGH}. More precisely, for a finite group $X(q)$ of Lie type, and a positive integer $m$ we will write $q_m$ for an arbitrary primitive prime divisor of $q^m-1$. In the table, we also use $r$ for the prime satisfying $q=r^f$, $f\in\mathbb{N}$. For a prime $p$, we will write $d_r(p)$ for the order of $r$ modulo $p$. We will also write $\mathcal{P}$ for the set of primes of the form $q^m-1/q-1$, with $q$ a prime power, $m\in\mathbb{N}$. Finally, using a slightly modified version of the notation in \cite{BBGH}, we will write $\alpha'(m,\epsilon)$ and $\beta'(m,\epsilon)$ for the conditions:
\begin{description}
    \item[$\alpha'(m,\epsilon)$] $q^{m/k} \not\equiv \epsilon \pmod{|R|}$ for all $k \in\pi(f)$.
    \item[$\beta'(m,\epsilon)$] $q^{m/k} \not\equiv \epsilon \pmod{|R|}$ for all odd primes $k \in\pi(f)$.
\end{description}
Here, $m\in\mathbb{N},\epsilon\in\{\pm 1\}$; $R$ is the weakly subnormal $p$-subgroup in question; and $p$ will be as indicated in the second column of the table.
   \end{rem}

\begin{table}[h!]
    \centering\small
    \renewcommand{\arraystretch}{1.3}
    \begin{tabular}{l l l l}
     $G$    &  $p$ & $M$ & Conditions\\
     \hline
     $\Alt_p$    &  $p$ & $p:\frac{p-1}{2}$ & $p\geq 13$, $p\neq 23$, $p\not\in\mathcal{P}$\\
     $\LL_2(q)$    &  $r$ & $r:\frac{r-1}{2}$ & $q=r$\\
                 &  $q_2$ & $D_{q+1}$ & $f\le 2$, and either $p > 5$, or $|R|>p$, or\\
                 &        &           & $f>2$ and $\alpha'(1,-1)$, or\\
                 &        &           & $f>2$, $(p, r) = (3, 2)$, and\\
    & & & $q^{1/k} \equiv 1 \pmod{|R|}$ for all $k \in\pi(f)-\{f\}$\\
    $\UU_3(q)$     &  $q_6$ & $\frac{1}{(q+1,3)}(q^2-q+ 1):3$ & $\beta'(3,-1)$ and either $f > 1$, or $|R| > p$, or $p > 7$\\
    $\LL_n(q)$     &  $q_n$ & $\frac{q^n-1}{q-1}:n$ & $n > 3$ prime, $\alpha'(n, 1)$ and either $f > 1$ is odd, or\\
    
& & & $f = 1$ and either $|R|>p$, or $p\neq 2n+ 1$, or \\
& & & $-p$ is a non-square modulo $r$\\
$\UU_n(q)$     &  $q_{2n}$ & $\frac{q^n+1}{q+1}:n$ & $n > 3$ prime, $\beta'(n,-1)$ and either $f > 1$, or\\
& & & $f = 1$ and either $|R|>p$, or $p\neq 2n+1$, or \\
& & & $-p$ is a square modulo $r$\\
${}^2{\rm B}_2(q)$    &  $q_4$ &  $q\pm\sqrt{2q}+1$ & $q^{2/k} \not\equiv -1 \pmod{|R|}$ for all odd $k \in\pi(f)-\{f\}$\\ 
${}^2{\rm G}_2(q)$    &  $q_6$ &  $q\pm\sqrt{3q}+1$ & $\alpha'(3,-1)$\\ 
${}^3{\rm D}_4(q)$    &  $q_{12}$ &  $q^4-q^2+1$ & $q^{6/k} \not\equiv -1 \pmod{|R|}$ for all odd $k \in\pi(f)-\{3\}$\\ 
${}^2{\rm F}_4(q)$    &  $q_{12}$ &  $q^2\pm\sqrt{2q^3}+q\pm\sqrt{2q}+1$ & $f\geq 3$ and $\alpha'(6,-1)$\\ 
${\rm E}_8(q)$    &  $q_{15(3-\epsilon)/2}$ &  $q^8-\epsilon q^7+\epsilon q^5-q^4+\epsilon q^3-\epsilon q+1$ & $\alpha'(30,1)$ and either $p>61$, or \\
& & &$|R|>p$, or $|R|=p=61$ and either $f > 2$, or \\
& & &$f = 2$, $i = 15$ and $d_{r}(p)\in \{15, 30\}$\\ 
$\mathrm{M}_{23}$    &  $23$ & $23:11$ & \\
$\mathrm{J}_{1}$    &  $19$ & $19:6$ & \\
$\mathrm{J}_{4}$    &  $29$ & $29:28$ & \\
                    &  $43$ & $43:14$ & \\
$\mathrm{Ly}$    &  $37$ & $37:18$ & \\
                 &  $67$ & $67:22$ & \\
$\mathrm{Fi}_{24}'$    &  $29$ & $29:14$ & \\
$\mathbb{B}$    &  $47$ & $47:23$ & \\
\hline
 \end{tabular}
    \caption{The pairs $(G,M)$ with $G$ a finite simple group containing a cyclic weakly subnormal $p$-subgroup $R$ with $\mathcal{M}(R)=\{M\}$.}
    \label{tab:thm3}
\end{table}

The main motivation for the study of weakly subnormal subgroups is to prove various variations of the Baer-Suzuki theorem. The Baer-Suzuki theorem states that if $p$ is a prime, $x$ is a $p$-element in a finite group $G$, and $\langle x,x^g\rangle$ is a $p$-group for all $g\in G$, then $x\in O_p(G)$. Many variations  of this theorem have been proved over the years (see \cite{GM1,GM2,GR,Wielandt}). In \cite{GR}, Guralnick and Robinson showed that if  $G$ is a finite group and $x\in G$ is an element of order $p$ such that $[x,g]$ is a $p$-element for every $g\in G,$ then $x\in O_p(G)$. Since $[x,g]=x^{-1}x^g\in\langle x,x^g\rangle$, this result (whose proof depends on the classification of finite simple groups) implies the Baer-Suzuki theorem. In fact, Guralnick and Malle \cite[Theorem 1.4]{GM1} prove a stronger result which says that if $x\in G$ is a $p$-element and $CC^{-1}$ consists of only $p$-elements, where $C=x^G$, then $C\subseteq O_p(G)$. They also conjecture that  if $p\neq 5$ is a prime and $C$ is a conjugacy class of $p$-elements in a finite group $G$ with $[c,d]$ a $p$-element for all $c,d\in C$, then $C\subseteq O_p(G)$ (see \cite[Conjecture 1.3]{GM1}).

In our first result, we prove the following variation of the Baer-Suzuki theorem.

\begin{theorem}\label{t:thmA}
Let $G$ be a finite group and let $p$ be a prime. Let $x\in G$ be a  $p$-element. Assume that $[x,g]$ is a $p$-element for every $p'$-element $g\in G$ of prime power order. Then $x\in O_p(G)$.
\end{theorem}

 Recall that if $g\in G$ is an element of a finite group $G$ and $p$ is a prime, then  $g$ is called a $p'$-element (or a $p$-regular element) if its order is coprime to $p$; it is called $p$-singular if its order is divisible by $p$.  Define $Z_p^*(G)$ to be a normal subgroup of $G$ containing $O_{p'}(G)$, the largest normal $p'$-subgroup of $G$, such that $Z^*_{p}(G)/O_{p'}(G)=Z(G/O_{p'}(G))$.  
 
 In the opposite direction to Theorem \ref{t:thmA}, Guralnick and Robinson proved a version of Glauberman's $Z_p^*$-theorem (\cite[Theorem D]{GR}) stating that if $x$ is an element of prime order $p$ of a finite group $G$ and $[x,g]$ is $p$-regular for every $g\in G$, then $x\in Z_p^*(G).$ It turns out that the condition $[x,g]$ is a $p'$-element for every element $g\in G$ of prime power order will be enough to guarantee the conclusion of the aforementioned theorem. 
 
\begin{theorem}\label{t:thmB}
Let $G$ be a finite group and let $p$ be a prime. Let $x\in G$ be a $p$-element. If $[x,g]$ is a $p'$-element for every element $g\in G$ of prime power order, then $x\in Z_p^*(G)$.
\end{theorem}

We cannot  assume that $[x,g]$ is $p'$-element for every $p'$-element $g\in G$ of prime power order (which is an exact opposite to Theorem \ref{t:thmA}). To see this, take $G=S_4$, the symmetric group of degree $4$  and $x$ any transposition in $G$. Then $[x,g]$ is a $3$-element for every $2'$- element $g\in G$ but clearly $x$ is not contained in $ Z_2^*(G)=1.$ Note that the hypothesis of  Glauberman's $ Z_p^*$-theorem implies that the element $x$ lies in the center of all Sylow $p$-subgroups of $G$ containing $x$. (See Theorem \ref{th:Glauberman} for other equivalent statements of Glauberman's $Z_p^*$-theorem).

We propose the following conjecture which is the one of strongest possible generalizations of the Baer-Suzuki theorem (as well
as Baer's theorem). Let  $k\ge 1$ be an integer and let $x\in G$ be a $p$-element.  Let $$\Gamma_k(x)=\{[g,{}_kx]:= [g,\underbrace{x,x,\dots, x}_{k \text{ times}}]: g\in G\},$$
where we define $[x_1,x_2,\dots,x_n]=[[x_1,x_2,\dots,x_{n-1}],x_n]$ for  $x_1,x_2,\dots,x_n\in G$ and any integer $n\ge 2$. 

\begin{conjecture}\label{conj: multicommutators}
Let $G$ be a finite group and let $p$ be a prime divisor of $|G|$. Let $x\in G$ be a $p$-element, and suppose that for some integer $k\geq 1$, $ab$ is a $p$-element for all $a,b\in \Gamma_k(x)$. Then $x\in O_p(G)$.
\end{conjecture}

For odd primes, this conjecture can be reduced to simple groups.  Note if $x \in A_5$ is an element of order $5$, then for $k > 1$,  $\Gamma_k(x)$ has size $6$
and consists of $5$ conjugates of $x$ and the identity element (see \cite{GM1}). Also, if $x\in \LL_2(8)$ has order $3$, then for $k > 1$,  $\Gamma_k(x)$ consists of $27$ elements of order $9$ and the identity element. Thus, for a $p$-element $x$, $\Gamma_k(x)$ consisting of $p$-elements does not guarantee that $x\in O_p(G)$ (at least for $p=3,5$). 

It would also be interesting to determine whether or not it is true that if $G$ is a finite group and $x\in G$ is a $p$-element such that for some integer $k\geq 1$, $[a,b]$ is a $p$-element for all $a,b\in \Gamma_k(x)$, then $x\in O_p(G)$.
Note that by \cite{GTr}, we have that if $\langle \Gamma_k(x) \rangle$ is a $p$-group,
then $\langle x \rangle$ is subnormal in $G$.

As an application of a generalization of the Baer-Suzuki theorem (\cite[Theorem A]{GR} and \cite[Theorem 1.4]{GM1}), it is proved in \cite[Theorem A]{BLMNST} that if $x\in G$ is a $p$-element, where $p$ is a prime and $G$ is a finite group, and $xy$ is a $p$-element for every $p$-element $y\in G$, then $x\in O_p(G)$. We prove a generalization of this result as follows.

\begin{theorem}\label{t:thmC}
 Let $G$ be a finite group and let $p$ be a prime. Let $x\in G$ be a $p$-element. Assume that $xy$ is either $1$ or $p$-singular for every $p$-element $y\in G$. Then $x\in O_p(G)$.
\end{theorem}

We do not know any counterexample to the following.

\begin{conjecture}\label{conjD}
Let $G$ be a finite group and let $p$ be a prime. Let $x\in G$ be an element of order $p$.   If $[x,g]$ is either $1$ or $p$-singular for every element $g\in G$, then $x\in O_p(G)$.
\end{conjecture}

The assumption on the order of $x$ is necessary since if $G=\GL_2(3)$ and $x\in G$ is an element of order $8$, then $[x,g]=1$ or  is $2$-singular for every $g\in G$ but $x\not\in O_2(G)$. Note that Conjecture \ref{conjD} is true when $G$ has a cyclic Sylow $p$-subgroup (\cite[Theorem 2.1]{GR}) or when $p=2$ (the Baer-Suzuki theorem). We show that Conjecture \ref{conjD} holds under the assumption that $O_p(G)$ is abelian (see Theorem \ref{t:conjDOp}) or the assumption that a Sylow $p$-subgroup of $G$ is abelian (see Corollary  \ref{c:conjDOp}).

We next complete the proof of the following result which is stated as Theorem E in \cite{BLMNST} modulo a conjecture about finite simple groups.
\begin{theorem}\label{t:thmE}
Let $G$ be a finite group and let $p$ be a prime. Let $x\in G$ be a $p$-element. Then $x\in O_p(G)$ if and only if $r$ divides o(xy) for all nontrivial $r$-elements $y\in G$ and all primes $r\neq p.$
\end{theorem} 

Finally in the last section, we present an application of Theorem \ref{t:thmC} to the character theory of finite groups.

\vs

\noindent \textbf{Acknowledgements.}   Guralnick was partially supported by the NSF grant DMS-1901595 and a Simons Foundation Fellowship 609771. Tracey was supported by the EPSRC Postdoctoral Fellowship EP/T017619/1.

\section{Weakly subnormal $p$-subgroups}\label{s:wsn}

In this section, we determine the structure of finite groups with a weakly subnormal $p$-subgroup for some prime $p$. In particular, we will prove Theorems \ref{t:psolvable},  \ref{t:main1} and \ref{t:cyclic}. Recall that a subgroup $R$ of $G$ is weakly subnormal in $G$ if $R$ is not subnormal in $G$ but $R$ is subnormal in all of its proper overgroups $H$ in $G$.  Let $\mathcal{M}(R)$ be the set of maximal subgroups of $G$ containing $R.$

\begin{lem} \label{l:r-subs} 
Let $G$ be a finite group and let $p$ be a prime divisor of $|G|.$ Let $R$ be a  weakly subnormal $p$-subgroup of $G$. Then the following hold.
\begin{enumerate}[{\rm(i)}]
\item $G$ is the normal closure of $R$;
\item $\mathcal{M}(R)=\{M\}$ and $R \le O_p(M)$;
\item $R$ is subnormal in $R\Phi(G)$ and $O_{p'}(G)\cap \Phi(G)\leq Z(G)$;
\item $O_{p'}(G) \le Z(G) \cap \Phi(G)$; or $G=QR$ where $Q\unlhd G$ is a $q$-group for some prime $q$ and $R$ acts irreducibly
on $Q/\Phi(Q)$ and centralizes $\Phi(Q)$; and
\item $F(G) \le M$ if the first case in (iv) holds.
\end{enumerate} 
\end{lem}

\begin{proof} Let $N$ be a normal subgroup of $G$ containing $R$. If $N\neq G,$ then  $R \le O_p(N) \le O_p(G)$ and so $R$ is subnormal in $G$, which is a contradiction. So (i) holds. 
Part (ii) follows from Wielandt's Zipper Lemma \cite[Theorem 2.9]{Isaacs}.    For part (iii), let $M$ be the unique maximal subgroup of $G$ containing $R$.  
Since $\Phi(G) \le M$ and $R$ is subnormal in $M$,  $R$ is subnormal in $R\Phi(G)$.    Moreover, since $[R,O_{p'}(G)\cap \Phi(G)]\leq [O_p(M),O_{p'}(M)]=1$, $O_{p'}(G) \cap \Phi(G)$ is centralized by $R$  and so by $G$.  Thus part (iii) holds.

If $R$ normalizes a $p'$-subgroup $Q$ but does not centralize it,  then 
$[Q,R,R]= [Q,R] \ne 1$ (see \cite[Lemma 4.29]{Isaacs}) and so $R$ is not subnormal in $[Q,R]R$, thus $G = [Q,R]R$.
By the theory of coprime group actions, there is a prime $q$ and an $R$-invariant Sylow $q$-subgroup $Q_0$ of
$Q$ satisfying the same condition and so we may assume that $Q$ is a $q$-group
and $G=QR$.   This is the case if $O_{p'}(G)>1$ is not central in $G$. Now the structure of $G=QR$ follows easily. In this case, $\Phi(G)=\Phi(Q)=Z(G)$ and $M=R\times \Phi(G)$.

Suppose that $O_{p'}(G)$ is central but not contained in the Frattini subgroup of $G$.   Then $G=O_{p'}(G)D$ for
some maximal subgroup $D$ of $G$.  Since $O_{p'}(G)$ is central in $G$, $D$ is normal in $G$.
 Since $D$ contains a Sylow $p$-subgroup of $G$, we may assume that
$R \le D$.  This implies that the normal closure of $R$ is contained in $D \ne G$, a contradiction.   Thus we have
proven (iv). 
Finally, since $R$ is subnormal in the $p$-subgroup $RO_p(G)$, $O_p(G) \le M$ and so (v) holds. 
\end{proof}

\begin{rem}\label{r:weakly subnormal}
Note that the converse of part (ii) holds, that is, a $p$-subgroup $R$ of $G$ is weakly subnormal in $G$ if and only if $\mathcal{M}(R)=\{M\}$; $R\leq O_p(M)$;
and $M$ is not  normal in $G$.
\end{rem}

\begin{rem}\label{r:Phi}
Let $R$ be a $p$-subgroup of a finite group $G$. 
\begin{enumerate}[{\rm(a)}]
\item Assume that $O_{p'}(G)$ is not central in $G$. Then $R$ is weakly subnormal in $G$ if and only if  $G$ is as described in the latter part of Lemma \ref{l:r-subs}(iv). In particular, $G$ is solvable.

\item Assuming that $O_{p'}(G)$ is central in $G$, then $R$ is weakly subnormal in $G$
if and only if $R\Phi(G)/\Phi(G)$ is weakly subnormal in $G/\Phi(G).$ 

\end{enumerate}
\end{rem}

Now Theorem \ref{t:psolvable} follows easily. 

\begin{proof}[\textbf{Proof of Theorem \ref{t:psolvable}}]
Let $G$ be a finite $p$-solvable group with $O_p(G)=1$. Assume that $R$ is a weakly subnormal $p$-subgroup of $G$.
Since $O_p(G)=1$ and $G$ is $p$-solvable, $F^*(G)$ is a nontrivial $p'$-group  and so $O_{p'}(G)$ is not central by Bender's theorem \cite[Theorem 9.8]{Isaacs}. Then by Lemma \ref{l:r-subs}(iii) and (iv), $G = QR$ with $Q=O_{p'}(G)$ a $q$-group, $x$ acting faithfully and irreducibly on $Q/\Phi(Q)$, and $\Phi(Q)\le Z(G)\le Q$. Since $(|x|,|Q|)=1$ and $x$ acts trivially on every proper $x$-invariant subgroup of $Q$, it follows that either $Q$ is elementary abelian, or 
$Q$ is special (i.e. $\Phi(Q)=[Q,Q]=Z(Q)$). The result follows.
\end{proof}

Next, we need the following result on the normalizers of Sylow subgroups of nonabelian simple groups.

 \begin{lem}  \label{l:selfnormal}  Let $S$ be a finite nonabelian simple group and let
  $p$ be  a prime dividing $|S|$.  Let $P$ be a 
 a Sylow $p$-subgroup of $S$.   If $p$ is odd, then $N_S(P) \ne P$.  If $p=2$, then one of the following holds:
\begin{enumerate}[{\rm(i)}]
 \item   $N_S(P) \ne P$; or
 \item   there exists an involution $z \in P$ with $z$ central in a Sylow $2$-subgroup of $\Aut(S)$ containing $P$ with
 $C_S(z) \ne P$; or 
 \item  $S \cong \Alt_6 \cong \LL_2(9)$; or
 \item   $S \cong \LL_2(r)$ with $r > 5$ a Fermat or Mersenne prime.
 \end{enumerate}
 \end{lem}
 
 \begin{proof}  If $p$ is odd, this  follows from \cite[Corollary 1.2]{GMN}.   Now assume that $p=2$.
 
 If $S$ is a sporadic group, then  this follows by inspection of the maximal subgroups of odd index (also by \cite{asmem}).
 Suppose that $S = \Alt_n, n \ge 5$.  If $n=5$, then (i) holds and if $n=6$, then (iii) holds.   If $n > 6$, then the centralizer of any involution
 in $S$ is not a $2$-group and so (ii) holds.
 
 Suppose that $S$ is a finite simple group of Lie type in characteristic $2$.    If (i) fails, then $S$ is defined
 over the prime field.   Considering centralizers of involutions (e.g. see \cite{LSbook}), we see that (ii) holds
 unless $S=\LL_3(2) \cong \LL_2(7)$.  
 
 Finally consider the case that $S$ is a finite simple group of Lie type over the field of $q$ elements with $q$ odd.
 Let $z \in P$ be an involution that is in the center of a Sylow $2$-subgroup of $\Aut(S)$ containing $P$.
 If $z$ is not regular semisimple, then $z$ centralizes unipotent elements and so (ii) holds.   If $z$ is regular
 semisimple, then $S \cong \LL_2(q)$ (and $z$ corresponds to an element of order $4$ in $\SL_2(q))$.
 If $q=5$, then (i) holds.   So assume $q > 5$. 
 Then $C_S(z)$ is the normalizer of a torus (split if $q \equiv 1 \mod 4$ and nonsplit otherwise).    Thus, 
 $C_S(z)$ is a $2$-group if and only if $q \pm 1$ is a power of $2$ and (iv) holds. 
 \end{proof}

 \begin{lem} \label{l:wreath}  Let a finite $p$-group $R$  act on a finite group $X= M_1 \times \ldots \times M_t$ with $M_i \cong M$
and $t > 1$.  Assume that $R$ transitively permutes the $M_i$.   Let $G=XR$.  Then $R$ is subnormal in $G$ if and only 
if $M$ is a $p$-group.
\end{lem}

\begin{proof}  If $M$ is a $p$-group, then so is $G$ and hence every subgroup of $G$ is subnormal. 
For the remaining, suppose that $M$ is not a $p$-group.  Then we may assume that $O_p(M)=O_p(X)=1$.   We can replace 
$M$ by a minimal characteristic subgroup and so assume that either $M$ is an $r$-group for some prime $r \ne p$
or $M$ is a nonabelian simple group.  In the first case, $[X,R,R]=[X,R]$ is a nontrivial $r$-group and so $R$ is not subnormal.
In the second case, $[X,R]=X$ since $[X,R] \le X$ is normal in $G$ and $X$ is a minimal normal subgroup of $G$. It follows as above that $R$ is not subnormal. 
\end{proof} 

\begin{lem} \label{l:E(G)}   Let $p$ be a prime and $G$ a finite group with $O_p(G)=1$.   Assume that $R$ is a weakly
subnormal $p$-subgroup of $G$ and that $G$ is not $p$-solvable.   Then $G=E(G)R$, $\Phi(G) \le E(G)$ and all
components of $G$ are conjugate.  Moreover one of the following holds:
\begin{enumerate}[{\rm(i)}]
\item  $E(G)$ is quasisimple; or 
\item  $p=2$,  and $R$ acts transitively on the components   and if $S$ is a component of $G$, then
then 
$S  \cong \LL_2(r)$ with $r > 5$ a Fermat or Mersenne prime or  $S \cong \LL_2(9)$  or 
$S$ is a triple cover of $\LL_2(9)$ and  $Z(E(G))=Z(G)$ has order $3$.
\end{enumerate} 
\end{lem}

\begin{proof}  Since $G$ is not $p$-solvable, by Lemma \ref{l:r-subs}, $O_{p'}(G)\leq Z(G)\cap \Phi(G)$. Furthermore, as $O_p(G)=1$,  it follows that $F(G) \le O_{p'}(G)=Z(G)\le \Phi(G)$, whence $F(G)=Z(G)=\Phi(G)$. Moreover,
 $E(G)$ is nontrivial and $p$ divides the order of every component of $G$. 

We will prove first that $G=E(G)R$.  If not, then $E(G)$ must be contained in the unique maximal subgroup of $G$ containing $R$ (say $M$). Let $D$ be a Sylow $p$-subgroup of $E(G)$ normalized by $R$.   Then 
$R \le E(G)R \ne G$ and $R \le N_G(D)$. Since $D$ is not normal in $G$, we must therefore have $N_G(D)\le M$. But $G=E(G)N_G(D)$ by the Frattini argument, which contradicts $E(G),N_G(D)\le M$.  The same argument show that $G=AR$, where $A$ is product of quasisimple groups normalized by $R$
and so all components of $G$ are conjugate. 
Since $G=E(G)R$ and $F(G)=Z(G)=\Phi(G)$, we deduce also that $\Phi(G)=Z(E(G))$.

Now $R$ normalizes $N:=N_{E(G)}(D)$ where $D$ is an $R$-invariant Sylow $p$-subgroup of $E(G)$.   Suppose that $E(G)$ is not quasisimple.
 By Lemma \ref{l:selfnormal},
if $p \ne 2$, $N/D$ is nontrivial. Since $E(G)$ is not quasisimple, Lemma \ref{l:wreath} implies that $R$ is not subnormal in $NR$.   

Similarly if $p=2$ and $E(G)$ is not quasisimple, then the same argument shows that an $R$-invariant Sylow $2$-subgroup $D$ of $E(G)$
is self normalizing.  Moreover, aside from the groups in the conclusions, by Lemma \ref{l:selfnormal}, there exists an involution
in $D\le E(G)$ which is centralised by $R$, and such that $C_{E(G)}(z)$ properly contains $D$. By Lemma \ref{l:wreath},  $R$ is not subnormal in 
$C_{E(G)}(z)R$. Thus, $G=C_{E(G)}(z)R$. Since $O_2(G)=1$, we have a contradiction.  Thus, the only possible components are odd central covers
of the simple groups listed in Lemma \ref{l:selfnormal}(iii) and (iv).  The only group with an nontrivial odd cover is $\Alt_6 \cong \LL_2(9)$.  If
the triple cover of $\Alt_6$ is a component, it follows that $Z(E(G)) \le Z(G)$ and so has order $3$. 
\end{proof}  

This now gives a classification of all groups containing a weakly subnormal $p$-group and the maximal such subgroups.

\begin{proof}[\textbf{Proof of Theorem \ref{t:main1}}]
Let $G$ be a finite group with $O_p(G)=1$. Assume that $G$ is not $p$-solvable and that $G$ has a weakly subnormal $p$-subgroup $R$. Let $M$ be the unique maximal subgroup of $G$ containing $R$. Let $P$ be a Sylow $p$-subgroup of $G$ containing $R$. Then $R\leq P\leq M$. Since $R$ is subnormal, we have $R\leq O_p(M)\leq P\leq M$.  As $O_p(G)=1$,  $M$ is the unique maximal subgroup of $G$ containing $P$ and $O_p(M)$. 

In view of Lemma \ref{l:E(G)}, if $E(G)$ is quasisimple,  we can pass to $G/\Phi(G)$, where $\Phi(G)=Z(E(G))$, and then apply the main results of \cite{BBGH}, (specifically, Corollaries $4$ and $6$). 
If $E(G)$ is not quasisimple, then $p=2$ and part (ii) of Lemma \ref{l:E(G)} holds yielding the last two cases of the theorem.
\end{proof}

  Note that if $R$ is a weakly subnormal $p$-subgroup with $M$ the unique maximal subgroup of 
$G$ containing $R$, then $R \le O_p(M)$, and $O_p(M)$ is also a weakly subnormal $p$-subgroup. Thus, we have classified all pairs $(G,R)$
where $R$ is a maximal weakly subnormal $p$-subgroup.

\begin{proof}[\textbf{Proof of Theorem \ref{t:cyclic}}]
As already noted  $R:=\langle x\rangle \le P \le M$ where $M$ is the unique maximal subgroup containing $R$, $P\in\rm{Syl}_p(G)$, and $R \le O_p(M)$.
If $G$ is $p$-solvable, then Theorem \ref{t:psolvable} applies.  So assume that $G$ is not $p$-solvable.  
One now has to check the cases in Theorem \ref{t:main1}. In the small cases, one checks the result directly (using GAP). We now discuss the infinite families coming from Theorem \ref{t:main1}. For ease of notation, we will assume (as we may) that $Z(G)=1$.  

Suppose first that $G$ lies in \cite[Table E]{BBGH}. If $G=A_p$ then the result is clear, so assume that $G$ is of Lie type. If $G$ has twisted rank greater than $1$ and $x$ is not regular semisimple, then $x$ is contained in at least two maximal subgroups of $G$ by \cite{GT}. Otherwise, $P$ is cyclic, and there is a unique conjugacy class of elements of order $|R|$ in $G$ (again, see \cite{GT}). Thus, $x$ is contained in a unique maximal subgroup of $G$ if and only if $M^G$ is the unique conjugacy class of maximal subgroups of $G$ with order divisible by $o(x)$, and $M$ is the unique conjugate of $M$ containing $x$. One can now combine the proofs in \cite[Section 6]{BBGH} and \cite[Tables 17--24]{GT} to deduce the conditions in Table \ref{tab:thm3}. For example, if $G=\LL_n(q)$ with $p=q_n$, then $n>3$ is prime and $f$ is odd by \cite[Table E]{BBGH}; while $|R|$ does not divide $q^{n/k}-1$ for any prime $k$ dividing $f$ by \cite[Table 17]{GT}. Further, we see from the proof of \cite[Proposition 6.2]{BBGH} that if $f=1$, then either $|R|>p$, or $p\neq (n-1)/2$, or $-p$ is a non-square modulo $r$. The remaining cases are entirely similar.

We now move on to the infinite families not in \cite[Table E]{BBGH}. If $p=2$, $G=\PGL_2(q)$ [respectively $G=\LL_2(q).2_3$] and $q$ is a Fermat prime [resp. the square of a Fermat prime], then every $2$-element of $G$ normalises a parabolic subgroup, whence is contained in at least two maximal subgroups by \cite{GT}.

If $p=2$ and $G=\LL_2(q)$ or $\PGL_2(q)$ with $q\equiv -1 \pmod{4}$ prime, then $|P|\geq 16$ by \cite[Corollary 6]{BBGH}, so $(q+1)_2\geq 8$ if $G=\PGL_2(q)$, and $(q+1)_2\geq 16$ if $G=\LL_2(q)$. Also, $|R|\geq 8$, since all elements of $G$ of order dividing $4$ are contained in a conjugate of a maximal $S_4$. Indeed, one can see from the list of maximal subgroups of $\LL_2(q)$ and $\PGL_2(q)$ that such a maximal $S_4$ subgroup always exists, since $q\equiv -1 \pmod{8}$. Thus, we see that $q\equiv -1 \pmod{8}$ and $|R|\geq 8$. 

The final case when $P$ is not cyclic is when $G$ is a rank $1$ simple group of Lie type in characteristic $p$.    
If $G =\LL_2(p^a)$, then we note that every $p$-element is contained  in a conjugate of $\LL_2(p)$  and so if $a > 1$, is not contained
in a unique maximal subgroup.  If $G=\UU_3(q)$ with $q$ odd, then it easy to see (or apply \cite{GT}) that every unipotent element  is conjugate
to an element of either $\textrm{SO}_3(q)$ or the stabilizer of a nondegenerate hyperplane.   If $q$ is even, then as $G$ is not solvable, $q \ge 4$.
But every element of order $4$ is conjugate to an element of $\UU_3(2)$ and so is not contained in a unique maximal subgroup.  
If $G$ is a Suzuki group, then any element of order $4$ normalizes a nonsplit torus.   If $G={^2}{\rm G}_2(3^a), a > 1$, then any unipotent element is conjugate to an element in ${^2}{\rm G}_2(3)$.   So the only examples are $\LL_2(p)$ with $p$ prime and $p \ge 5$.  This completes the proof of the theorem.
\end{proof}

\begin{proof}[\textbf{Proof of Corollary \ref{c:cyclic34}}] Let $G$ and $R$ be as in the statement of the corollary, and assume that $G$ is insolvable. The case $|R|=2$ is clear so assume first that $|R|=3$. Then by Theorems \ref{t:main1} and \ref{t:cyclic}, $G/O_3(G)$ is isomorphic to $\LL_2(2^e)$, with $e$ odd. If $e$ is not prime, then $|R|=3$ divides $|\LL_2(2^{e/k})|$ for all prime divisors $k$ of $e$. Since all elements of order $3$ are conjugate in $|\LL_2(2^e)|$ in this case, we see that $R$ is contained in more than one maximal subgroup -- a contradiction. Thus, $e$ is an odd prime, as needed. 

Suppose next that $|R|=|\langle x\rangle|=4$. Then by Theorem \ref{t:cyclic}, we have $\overline{G}=E(\overline{G})R\le A\wr \langle \sigma\rangle$, where $o(\sigma)\in\{2,4\}$, and $A\in\{\PGL_2(7),M_{10},\LL_2(q),\PGL_2(q)\}$ with $q>7$ a Mersenne prime. Further, $|Z(G)|$ divides $3$, and $N_{\overline{G}}(\soc(A))/C_{\overline{G}}(\soc(A))\cong A$. 
It follows that $\overline{x}=(y_1,\hdots,y_t)\sigma$, where $y:=\prod_iy_i$ is a $2$-element of $A$ which generates $\soc(A)/A$ modulo $A$. Since $y:=\prod_iy_i$ has order $o(\overline{x})/o(\sigma)$ and $o(\overline{x})=4$, we must have $y=1$ or $o(y)=2$, and $G\neq M_{10}$. 
By replacing $\overline{G}$ by an $\Aut(\overline{G})$-conjugate, we may assume that  if $y=1$, then $\overline{x}=\sigma$; while if $o(y)=2$, then $\overline{x}=(y_1,1)\sigma$. Clearly $\overline{x}$ is not contained in a unique maximal subgroup in the former case, so we may assume that $\overline{x}=(y_1,1)\sigma$, with $|y_1|=2$. Then in each of the cases $A\in\{\PGL_2(7),\LL_2(q),\PGL_2(q)\}$, $y_1$ normalises at least two maximal subgroups $M_1$ and $M_2$ of $\soc(A)$. Thus, $\overline{x}$ lies in the distinct maximal subgroups $N_{\overline{G}}(M_1^2)$ and $N_{\overline{G}}(M_2^2)$ of $\overline{G}$. This final contradiction completes the proof. \end{proof} 
  
We close this section which yields some information for groups with more than one component.

\begin{lem} \label{l:wreath2}   Let $G$ be a finite group and let $Q$ be a component of $G$.
Suppose that $x \in G$ does not normalize $Q$.   If $r$ is any prime dividing $|Q|$, there exists
an $r$-element $y \in E(G)$ with $[x,y]$ a nontrivial $r$-element.
\end{lem}  

\begin{proof}  
There is no loss of generality in assuming that $E(G)$ is a central product of the conjugates of $Q$, and that $x$  permutes the conjugates of $Q$ transitively. 
It then follows that $x$ induces an automorphism of $E(G)$ of the form $a \rho$ where $a$ normalises $Q$, and
$\rho$ permutes the conjugates of $Q$ in a cycle of length $s \ge 2$.

If   $b \in Q^x$,  then $[x,b] = b^{-a\rho}b$.   Since $b^{-a\rho}$ and $b$ are contained
in distinct components,  we see that if $b$ is an $r$-element,  then $[x,b]$ is a nontrivial $r$-element.
\end{proof}

\section{Reduction results  for Baer-Suzuki type problems}\label{s:reduction}

Let $p$ be a prime and let $G$ be a finite group. Let $x\in G$ be a $p$-element. Let $\mathcal{P}$ be a property of the pair $(G,x)$ such that if $H$ is any subgroup of  $G$ containing $x$, then the pair $(H,x)$ also satisfies property $\mathcal{P}$. We call such property a Baer-Suzuki property.

We call the following problem a Baer-Suzuki type problem $\PP.$

\begin{problem} If the pair $(G,x)$ satisfies the Baer-Suzuki property $\PP$, then $x\in O_p(G)$.

\end{problem}

Since $O_p(G)$ is nilpotent, $x\in O_p(G)$ if and only if $\langle x\rangle$ is subnormal in $G$. 
Suppose that the pair $(G,x)$ is a counterexample to the Baer-Suzuki type problem $\PP$ as above with $|G|$ minimal. Then $x\in O_p(H)$ for every proper subgroup $H$ of $G$ containing $x$ but $x\not\in O_p(G)$. In other words,  the cyclic subgroup $\langle x\rangle$ is weakly subnormal in $G$. By Wielandt's zipper lemma, $G$ has a unique maximal subgroup, say $M$, containing $x$.

If the Baer-Suzuki property $\PP$ satisfies an additional condition that the pair $(G,x)$  satisfies $\PP$ if and only if the pair $(G/O_p(G),xO_p(G))$ satisfies $\PP$, then  we may assume that $O_p(G)=1$. In this situation, we can apply results in Theorems \ref{t:psolvable} and \ref{t:cyclic} to determine the structure of $G$.

\begin{prop}\label{p:reduction}
Let the pair $(G,x)$ be a counterexample to the Baer-Suzuki type problem $\PP$ with $|G|$ minimal. Assume that $O_p(G)=1.$ Let $M$ be the unique maximal subgroup of $G$ containing $x$. Then  $G$ is either solvable and the structure of $G$ is given in Theorem \ref{t:psolvable} or  $G$ is not $p$-solvable and  one of  the following holds.

\begin{enumerate}[{\rm (1)}]
\item If $p>5$, then $G$ is quasisimple, a Sylow $p$-subgroup of $G$ is cyclic, $\langle x\rangle$ is any nontrivial  $p$-subgroup and $M=N_G(\langle x\rangle)$. 
Moreover, $(G/Z(G),M/Z(G))$ is given in Table \ref{tab:thm3}.

\item If $p=5$, then either $G$ is described as in \emph{(1)} or $G={}^2\textrm{B}_2(32).5$, $\langle x\rangle$ is a cyclic group of order $25$ not contained in the socle, and $M$ is the normalizer of a nonsplit torus of order $25$.

\item If $p=3$, then $G$ is as in  \emph{(1)} or $G=\LL_2(8).3$, $\langle x\rangle$ is any cyclic  group of order $9$ not contained in the socle and $M$ is the normalizer of the nonsplit torus of order $9.$

\item Assume $p=2$. Then  one of the following cases holds.
\begin{enumerate}[{\rm(i)}]

\item $p=2$, $G={\rm M}_{10}$, $\langle x\rangle$ is any group of order $8$ not contained in the socle, and $M$ is a Sylow $2$-subgroup;  or 
\item $p=2$,  $G=\LL_2(q)$ or $\PGL_2(q)$, $M$ is the normalizer of a nonsplit torus, $q$ is prime, $q \equiv 3 \pmod 4$ and  $o(x) \ge 16$; or

\item $p=2$,  $G=E(G)\langle x\rangle$ and $E(G)  = T_1 \times \ldots \times T_t, t > 1$ is a minimal normal subgroup 
and if $T = T_1$, then $N_G(T)/C_G(T)$ has a maximal Sylow $2$-subgroup and  $N_G(T)/C_G(T)$ is isomorphic to one of
${\rm PGL}_2(7), \; {\rm M}_{10},  \; {\rm L}_2(q),\; {\rm PGL}_2(q),$
where $q >7$ is a Mersenne prime; or
\item  $p=2$, $G=E(G)\langle x\rangle$ and $E(G)$ is a central product of triple covers of $\Alt_6 = \LL_2(9)$,  $E(G)$ has a center
of order $3$ and if $T$ is a component of $G$, then $N_G(T)/C_G(T) = {\rm M}_{10}$.  
\end{enumerate}
\end{enumerate}
\end{prop}

\begin{proof}
This follows from Theorem \ref{t:psolvable} for $p$-solvable groups  and Theorem \ref{t:cyclic} for not $p$-solvable groups and the discussion above. Notice that a quasisimple group cannot have a cyclic Sylow $2$-subgroup. 
\end{proof}
 
 \vs
  
 There are certain conditions in which it is not clear that one can assume that $O_p(G)=1$.  If we impose an extra condition on the Sylow $p$-subgroup, then
we can say more.  Thus, 
we need the following results about groups with abelian Sylow $p$-subgroups.  

\begin{thm} \label{t:abelian}  Let $p$ be a prime.  Suppose that $G$ is a finite group with an abelian Sylow $p$-subgroup $P$
and $G = \langle P^g\text{ : }g \in G \rangle$.   Then $O_p(G)$ is central.   If $O_{p'}(G)$ is central,  then 
$G = O_p(G) \times E(G)$ with every component of $G$ having order divisible by $p$.  In particular,  $Z(E(G))=O_{p'}(G)$.
\end{thm}

\begin{proof}  Since $O_p(G)=\bigcap_{g\in G}P^g$ and $G= \langle P^g\text{ : }g \in G \rangle$, it is clear that $O_p(G)\le Z(G)$. For the remainder of the proof, suppose $O_{p'}(G)\le Z(G)$. Then $F(G)= Z(G)$.  
We claim that $O_{p'}(G) \le \Phi(G)$.   If not, then by Gasch\"utz theorem,   $G/\Phi(G) = A \times L$ where $A=O_{p'}(G)/(\Phi(G) \cap O_{p'}(G)$ and
$L \cap A =1$.  Since $L$ contains a Sylow $p$-subgroup of $G/\Phi(G)$,  $G=L$ as required.

First assume that $p$ is odd.  Then by \cite[Corollary 1.2]{GGLN},  $P= Z(P) \le F^*(G)$  (in that result, it is assumed that $O_{p'}(G)=1$ but
it is clear that all that is required is that $O_{p'}(G)$ is central).  Thus $P = Z(P) \le F^*(G)$ and so
$G=F^*(G) = O_p(G)E(G)$ (since $O_{p'}(F(G)) \le \Phi(G))$.   By inspection of
the covering groups
of the simple groups with abelian Sylow p-subgroup (see \cite[Section 6.1]{GLS} and \cite{ShenZhou}),  $Z(E(G))$ is a
$p'$-group.

Now assume that $p=2$. The only simple groups $S$ with abelian Sylow $2$-subgroups are ${\rm J}_1$, ${}^2{\rm G}_2(3^a)$ with $a$ odd,  $\LL_2(q)$ with $q = 2^a \ge 4$
or $\LL_2(q)$ with $q \equiv \pm 3 \mod 8$ \cite{Walter}.   One then observes that if $X$ is any quasisimple group with $X/Z(X)\cong S$ and $|Z(X)|$ even, the Sylow $2$-subgroups of $X$ are nonabelian. It follows that $Z(E(G))$ has odd order. Thus, all that remains is to prove that $G=O_2(G)E(G)$. Since $O_{p'}(G)$ is Frattini, it suffices to prove that $G=F^*(G)$. If not, then since $G/F^*(G)$ is generated by Sylow $2$-subgroups, there exists an element $x$ of $G\setminus F^*(G)$ of $2$-power order. Since $F(G)=Z(G)$, $F^*(G)$ contains its centralizer, and $G$ has abelian Sylow $2$-subgroups, such an element normalizes each component $Q$ of $G$, and induces a non-trivial outer automorphism of $S:=Q/Z(Q)$. One can check from the list of possibilities for $S$ above that a Sylow $2$-subgroup of $\langle S, \alpha \rangle$ is nonabelian for any outer automorphism $\alpha$ of $S$ of even order. This final contradiction yields the result.
\end{proof}

If there is a weakly subnormal $p$-subgroup, we can say more.

\begin{cor} \label{c:abelian}  Let $p$ be a prime.  Suppose that $G$ is a finite group with an abelian Sylow $p$-subgroup $P$.
Suppose that $R$ is a weakly subnormal $p$-subgroup of $G$.  Then one of the following holds:
\begin{enumerate}[{\rm(i)}]
\item $O_{p'}(G)$ is non-central and $G = QR$ with $Q=O_{p'}(G)$ and $R/O_p(G)$ acting irreducibly and faithfully on $Q/\Phi(Q)$; or 
\item  $G = O_p(G) \times Q$ with $Q$ quasisimple and $Z(Q)=O_{p'}(G)$.
\end{enumerate} 
\end{cor}

\begin{proof}   By Lemma \ref{l:r-subs}, $G = \langle P^g\text{ : } g \in G\rangle$, so the previous theorem applies. If $O_{p'}(G)$ is not central, then Theorem \ref{t:psolvable} implies (i). 
If $O_{p'}(G)$ is central,  then the previous result implies that $G = O_p(G) \times E(G)$.   By Theorem \ref{t:main1},   $E(G)/Z(E(G))$
is a minimal normal subgroup of $G$.  Since $P$ is abelian and every component has order a multiple of $p$,  each component is
normal and so $E(G)=Q$ is quasisimple and the result follows. 
\end{proof} 

 We now obtain a reduction result for Baer-Suzuki type problem when a Sylow $p$-subgroup is abelian.

\begin{prop}\label{p:abelian}
Let the pair $(G,x)$ be a counterexample to the Baer-Suzuki type problem $\PP$ with $|G|$ minimal. Let $P$ be a Sylow $p$-subgroup of $G$ containing $x$. Assume that $P$ is abelian. Then $O_p(G)\leq Z(G)$ and one of  the following holds.

\begin{enumerate}[{\rm(i)}]
\item $O_{p'}(G)$ is non-central and $G = QR$ with $Q=O_{p'}(G)$ and $R/O_p(G)$ acting irreducibly and faithfully on $Q/\Phi(Q)$; or 
\item  $G = O_p(G) \times Q$ with $Q$ quasisimple and $Z(Q)=O_{p'}(G)$. Moreover, $p>2$, and $Q$ is described in Theorem \ref{t:cyclic}(i).
\end{enumerate} 
\end{prop}

\begin{proof}  By Lemma \ref{l:r-subs}(i), $G=\langle x^g: g\in G\rangle=\langle P^g: g\in G\rangle$  so $O_p(G)$ is central by  Theorem \ref{t:abelian}.  By Corollary \ref{c:abelian}, the proposition follows apart from the last claim in part (ii).
Since $P$ is abelian, the Sylow $p$-subgroup $P\cap Q$ of $Q$ is abelian. Since  $\langle x\rangle O_p(G)/O_p(G)$ is weakly subnormal in $G/O_p(G)\cong Q$ and $O_p(Q)=1$, so $Q$ is one of the quasisimple groups appearing in Theorem \ref{t:cyclic}.

If $p>2$, then clearly $Q$ is in Case (i) of Theorem \ref{t:cyclic}. 
Next, assume that $p=2$. By inspecting cases (iv) - (vii) in Theorem \ref{t:cyclic}, the only possibility is $Q=\LL_2(q)$, $q$ is a prime with  $q\equiv 3$ mod $8$  and the order of $xO_p(G)$ in $G/O_p(G)\cong Q$ is at least $16$.  However, this cannot occur since the Sylow $2$-subgroup of $\LL_2(q)$ has order $4$.
\end{proof}
\section{Applications to Baer-Suzuki type problems}\label{s:applications}

We apply the reduction results in the previous sections to solve several Baer-Suzuki type problems. Let $G$ be a finite group and let $p$ be a prime. Let $P$ be a Sylow $p$-subgroup of $G$ and let $x\in P.$

We first consider the following property for the pair $(G,x)$ $$(\PP_1): [x,g] \text{ is a $p$-element for every $p'$-element $g\in G$ of prime power order}.$$  
Clearly $\PP_1$ is a Baer-Suzuki property as if $H$ is any overgroup of $\langle x\rangle$ in $G$, then the pair $(H,x)$ also satisfies property $\PP_1.$
The following easy lemma will show that the pair $(G,x)$ satisfies property $\PP_1$ if and only if $(G/O_p(G),xO_p(G))$ does.

\begin{lem}\label{l:prime powers}
Let $G$ be a finite group. Let $N\unlhd G, H\leq G$ and $g\in G.$ Then 

\begin{enumerate}[{\rm(i)}]
\item If $Ng\in G/N$ is an $r$-element for some prime $r$, then $Ng=Ny$ for some $r$-element $y\in G.$ 

\item If $g$ centralizes every element of prime power order of $H$, then $g$ centralizes $H.$
\end{enumerate}
\end{lem}

\begin{proof}
 Let $Ng\in G/N$ 
 be a nontrivial $r$-element for some odd prime $r$. Let $r^a=o(Ng)$ for some integer $a\ge 1.$ Assume that $o(g)=r^bm$ for some integers $b,m\ge 1$ with $r\nmid m.$ Then $b\ge a$ and so  $g^{r^b}\in N$ since $g^{r^a}\in N$. As $\gcd(r^b,m)=1$,  there exist integers $u,v$ such that $1=ur^b+vm$. Let $y=g^{vm}$. Then $y\in G$ is an $r$-element and $Ng=Ny$. This proves (i). For (ii), observe that  every element of $H$ can be written as a product of elements of prime power order. The proof of the lemma is complete.
\end{proof}

To justify the claim above, let $g\in G$ and let  $N=O_p(G)$. Assume that $Ng$ is $p'$-element of prime power order. By Lemma \ref{l:prime powers}(i), $Ng=Ny$ for some $y\in G$ for some $p'$-element of prime power order. Thus $[Nx,Ng]=[Nx,Ny]=N[x,y]$ is a $p$-element since $[x,y]$ is a $p$-element. Consequently, the pair $(G/O_p(G),xO_p(G))$ satisfies property $\PP_1.$ The converse is clear.

  We now prove Theorem \ref{t:thmA} which  generalizes \cite[Theorem A]{GR}. 

\begin{proof}[\textbf{Proof of Theorem \ref{t:thmA}}]
Let the pair $(G,x)$ be a counterexample to Theorem \ref{t:thmA} with $|G|$ minimal. By the discussion above, we can assume that $O_p(G)=1$. Then the structure of $G$ is given in Proposition \ref{p:reduction}. Let $P$ be a Sylow $p$-subgroup of $G$ containing $x$. We consider the following cases.

Assume $G$ is $p$-solvable.  By Theorem \ref{t:psolvable}, $G=Q\langle x\rangle$, where $Q$ is a normal $q$-subgroup of $G$ for some prime $q\neq p$, $\langle x\rangle$ acts irreducibly on $Q/\Phi(Q)$ and centralizes $\Phi(Q)$. Since $x$ does not centralize $Q$, there exists $y\in Q$ such that $[x,y]\neq 1$ and clearly $[x,y]\in Q$ is a $q$-element which is a contradiction.

So we assume that $G$ is not $p$-solvable. By Lemma \ref{l:r-subs}(iv), $O_{p'}(G)$ is central in $G$. Now by Lemma \ref{l:prime powers}, $(G/O_{p'}(G),xO_{p'}(G))$ satisfies property $\PP_1$, so if $O_{p'}(G)\neq 1$, then $xO_{p'}(G)\in O_p(G/O_{p'}(G))$ by the minimality of $|G|.$ However, as $O_{p'}(G)\leq Z(G)$ and $O_p(G)=1$, it is easy to see that $O_p(G/O_{p'}(G))=1$, a contradiction. Thus we can assume that $O_{p'}(G)=1.$
We now consider the case when $P$ is abelian or nonabelian separately.

Assume $P$ is nonabelian.   Suppose that $p$ is odd.  Then   $p=5$ and $G={^2}{\rm B}_2(32).5$ or   $p=3$ and $G=\LL_2(8).3$.   One computes directly that the result holds in these cases.

Suppose that $p=2$.   If $G$ has more than one component, the result follows by Lemma \ref{l:wreath2}.   So we may assume that $G$ is almost simple.
Inspecting Proposition \ref{p:reduction}(4) leads to the cases (i) and (ii). A straightforward computation settles (i). For (ii), suppose that $G=\LL_2(q)$ or $\PGL_2(q)$, with $q$ prime and $q\equiv 3\pmod{4}$. Then $x$ normalizes a nonsplit torus, and can therefore be lifted (up to conjugation) to a matrix of the form
$$\hat{x}=\begin{pmatrix}
0 & \pm 1\\
-1 &t
\end{pmatrix}.$$
One can check that the commutator $[x,\mathrm{diag}(y,y^{-1})]$ yields a nontrivial upper triangular matrix with eigenvalues $y^2$ and $y^{-2}$ if $y\neq\{\pm 1\}$. Since the Borel subgroup of $\LL_2(q)$ has odd order, the result follows.
Finally, if  $P$ is abelian, then the result follows from Theorem \ref{t:abeliansylow} below. The proof  is now complete.
\end{proof}

Note that the previous result generalizes \cite[Theorem 2.1]{GR} where it was assumed that the Sylow $p$-subgroup is cyclic and the conclusion is that $[x,g] \ne 1$ is a $p'$-element for some $g \in G$. 

\begin{thm}\label{t:abeliansylow}
Let $p$ be a prime and let $G$ be a finite group with an abelian Sylow $p$-subgroup $P$. Let $x\in P$ be a $p$-element. Then either $x\in  O_p(G)$ or there exists a prime $r\neq p$ and an $r$-element $y$ such that $[x,y]$ is a nontrivial $p'$-element.
\end{thm}

\begin{proof}
Let $G$ be a minimal counterexample to the theorem. Then  $[x,y]=1$ or $[x,y]$  is $p$-singular for every $p'$-element $y\in G$ of prime power order and $x\not\in O_p(G)$. It follows that if $x\in H\lneq G,$ then $x\in O_p(H)$ and so $\langle x\rangle$ is subnormal in $H$.   Thus, $\langle x \rangle$ is weakly subnormal in $G$.   If $O_{p'}(G)$ is non-central, then by Corollary \ref{c:abelian}(i) there exists a prime $r \ne p$
so that $x$ normalizes but does not
centralize some $r$-subgroup of $G$ and the result follows.  So we may assume that $O_{p'}(G)$ is central and 
by Corollary \ref{c:abelian}(ii),  $G = O_p(G) \times Q$ with $Q$ quasisimple and $Z(Q)=O_{p'}(G)$.   Clearly, we can assume that $O_p(G)=1=O_{p'}(G)$ and so we may assume that
$G$ is simple. Moreover, Theorem \ref{t:cyclic} applies. 

We go through the possibilities. 

\smallskip
(1) $G$ is not a sporadic simple group nor an alternating group of degree $n\ge 5$.
 
 If $G$ is alternating or a sporadic simple group, then by Theorem \ref{t:cyclic} the possibilities for $(G,p)$ are given in Table \ref{tab:thm3}. We can check using GAP \cite{GAP} that if $G$ is sporadic, then there exists an $r$-element $y\in G$ such that $[x,y]$ is a nontrivial $p'$-element. If $G= \Alt_n$ then $n=p\ge 5$
 and we can choose a $3$-cycle so that $[x,y]$ has order $3$.

  \smallskip
 (2) $G$ is a finite simple group of  Lie type in characteristic $\ell\neq p.$
  
 Assume by contradiction that $\ell=p$. Then $G=\PSL_2(q)$ with $q = p^a > 5$ is the only possibility as $G$ has an abelian Sylow $p$-subgroup.  If $p \ne 2$, direct calculation shows that $[x,g]$ can have arbitrary trace for $g \in G$ an involution and in particular can be an element of order $3$.  If $p=2$, these groups do not have weakly subnormal cyclic $2$-subgroups.
 
  \smallskip
 (3) $x$ is a semisimple element and if $R$ is a parabolic subgroup of $G$, then $p$ does not divide $|R|$ and $C_R(x)=1$.     
 
 Clearly, $x$ is semisimple as $p\neq \ell.$  The claim now  follows from 
 Theorem \ref{t:cyclic}.

\smallskip
 (4)   $G$ has (twisted) Lie rank $\ge 2$.   

 Assume that $G$ has (twisted) Lie rank $1$. Let $B$ be the Borel subgroup of $G$. Then $G$ acts doubly transitively on $\Omega=G/B.$  By (2) and (3),  $p$ does not divide $|B|$ and so $x$ has no fixed points on $\Omega$. Now let $r\neq \ell$ be a prime divisor of $|B|$ such that $r$ does not divide $|C_G(x)|$, where $ C_G(x)$ is a maximal torus of $G$. Let $y\in B$ be a nontrivial semisimple $r$-element which has at least two fixed points on $\Omega$. It follows that $[x,y]$ is nontrivial.
  Suppose that $x\cdot \alpha=\beta$ for some $\alpha,\beta\in \Omega$. Note that  $\alpha\neq \beta$. Since $G$ is $2$-transitive and $y$ has two fixed points, we may assume that $y$ fixes both $\alpha$ and $\beta.$ This implies that $[x,y]$ fixes $\alpha$ and 
  thus the order of $[x,y]\neq 1$ is coprime to $p.$

\smallskip
(5)   The theorem holds.

Essentially the same argument as given in (4) applies.   Let $R$ be a maximal parabolic subgroup of $G$.   Note that since $G$ has rank at least $2$,  $R \cap R^x \ne 1$
(indeed $R \cap R^g \ne 1$
for any $g \in G$).    Thus, by (3), we can choose an $r$-element $y$ of $R\cap R^x$, with $r\neq p$. Then $y$ fixes the points $\alpha:=R$ and $\beta:=xR\in\Omega:=G/R$. As above, $[x,y]$ therefore fixes $\alpha$.   Thus, again by (3), $[x,y]$ is a nontrivial $p'$-element of $G$.     
 \end{proof}

\vs

Consider the following property for the pair $(G,x)$, where $x\in G$ is a $p$-element.

$$(\PP_2): xy \text{ is $1$ or  $p$-singular for every $p$-element $y\in G$}.$$  

 Let $y\in G$. Observe that $(xy)^x=x^{-1}(xy)x=yx$. Hence $xy$ and $yx$ have the same order. 
Therefore, if $xy$ is either 1 or $p$-singular for every $p$-element $y\in G$, then $yx$ is either $1$ or $p$-singular for every $p$-element $y\in G.$ Now if the pair $(G,x)$ satisfies property $\PP_2$, then for every $g\in G$, we have $[x,g]=x^{-1}g^{-1}xg=(x^{g})^{-1}x$ is either $1$ or $p$-singular since $(x^{g})^{-1}$ is a $p$-element. It is clear that $\PP_2$ is a Baer-Suzuki property.

We claim that the pair $(G,x)$ satisfies $\PP_2$ if and only if $(G/K,Kx)$ satisfies $\PP_2$, where $K=O_p(G)$.

Assume first that $(G,x)$  satisfies $\PP_2$. 
We may assume that $x\not\in K$. Note that $ O_p(G/K)$ is trivial. Let $Ky\in G/K$ be a $p$-element. Then $y\in G$ is a $p$-element and thus $xy$ is either $1$ or $p$-singular. Assume that $KxKy=Kxy$ is a nontrivial $p'$-element. 
Write $xy=az=za$, where $a$ is a $p$-element and $z$ is a nontrivial $p'$-element. Then  $Kxy=Kaz=KaKz$ is a $p'$-element. Since $Ka$ and $Kz$ commute, we deduce that $Ka=K$ and hence $a\in K.$ It follows that $x(ya^{-1})=z$, where $ya^{-1}\in \langle y\rangle K$ is a $p$-element. However, this violates the $\PP_2$ property. Thus $KxKy\in G/K$ is either $1$ or $p$-singular for every $p$-element $Ky\in G/K$.

Conversely, assume that $(G/K,Kx)$ satisfies $\PP_2$. Let $y\in G$ be a $p$-element. Assume that $xy\neq 1$ is a $p'$-element. Then $Kxy=KxKy$ is a $p'$-element in $G/K$. By the assumption, $Kxy=K$ or $xy\in K $ is a $p$-element. Since $xy$ is a $p'$-element, we must have $xy=1$, a contradiction.

\begin{proof}[\textbf{Proof of Theorem \ref{t:thmC}}]
Let the pair $(G,x)$ be a counterexample to the theorem with $|G|$ minimal. Then $\langle x\rangle$ is weakly subnormal in $G$ and by the discussion above, we may assume $O_p(G)=1$, so Proposition \ref{p:reduction} applies.

 Next, we claim that $N= O_{p'}(G)=1.$ Suppose by contradiction that $N>1$. Clearly $Nx\in G/N$ is a $p$-element. Now let $Ng\in G/N$ be a $p$-element of $G/N$. By Lemma \ref{l:prime powers}, we may assume that $g\in G$ is a $p$-element and thus $xg$ is either $1$ or $p$-singular. Since $N$ is a $p'$-group, we see that $Nx\cdot Ng=Nxg$ is either $1$ or $p$-singular. Since $|G/N|<|G|$, by the minimality of $|G|$, we have $Nx\in O_p(G/N)$. Let $K$ be a normal subgroup of $G$ containing $N$ such that $K/N= O_p(G/N)$. Then $x\in K\unlhd G$ which forces $K=G$ as $G=\langle x^G\rangle$ by Lemma \ref{l:r-subs}. For any $n\in N,$ we have $x(n^{-1}x^{-1}n)=[x^{-1},n]\in N$ is a $p'$-element. Since $(n^{-1}x^{-1}n)\in G$ is a $p$-element, we must have that $[x^{-1},n]=1$ and so $[x,N]=1$. As $G=\langle x^G\rangle,$ $G$ centralizes $N$ and thus $x\in O_p(G)$, a contradiction. Therefore, we can assume that $ O_{p'}(G)=1$. It follows that $G$ is not $p$-solvable and thus one of the cases (1)-(4) in Proposition \ref{p:reduction} holds.

Let $N$ be a minimal normal subgroup of $G$. Since $ O_p(G)= O_{p'}(G)=1$, $N=T_1\times T_2\times\cdots\times T_t$, where each $T_i$ is conjugate in $G$ to $T=T_1$,  a non-abelian simple group with $p$ dividing $|T|$, and $k\ge 1$ is an integer. Assume that $\langle x\rangle N\neq G.$ Then $N\leq M$ and thus $[x,N]=1$ as $x\in O_p(M)$. It follows that $x\in C_G(N)\unlhd G$ and since $G=\langle x^G\rangle,$ $G= C_G(N)$ which forces $N\leq  Z(G)$, a contradiction. Thus $G=\langle x\rangle N.$ Note that $\langle x\rangle$ acts transitively on the simple factors $\{T_i\}_{i=1}^t$ by conjugation. 

Assume that $t\ge 2.$ Let $r\neq p$ be a prime that divides $|T_1|$ and let $R\in\Syl_r(T_1)$. Assume that $T_1^x=T_j$ for some $j\neq 1.$ Assume that  $x$ does not centralizes $R$. Then there exists $y\in R$ with $y\neq y^x$. Then $y^x\in T_j$ commutes with $y$. Hence $$y^{-1}y^x=y^{-1}x^{-1}yx=(x^y)^{-1}x=x(x^{yx})^{-1}$$ is an $r$-element. Since $(x^{yx})^{-1}$ is a $p$-element, $y^{-1}y^x=1$ or $y^x=y$, a contradiction.

So $k=1$ and $G$ is almost simple with socle $T$.   If the Sylow $p$-subgroup of $G$ is abelian, then Theorem \ref{t:abeliansylow} applies (note $[x,y]$ is a product of two $p$-elements).  
This leaves only one case each for $p=3$ and $5$ which are easy to check.  The cases with $p=2$ (i.~e. case (v) in Theorem \ref{t:cyclic}) have socle $\LL_2(q)$, with $q$ prime, $q\equiv 3\pmod{4}$ and $x$ normalizing a nonsplit torus. As in the proof of Theorem \ref{t:thmA}, $x$ can be lifted (up to conjugation) to a matrix $\hat{x}$ such that $[\hat{x},\mathrm{diag}(y,y^{-1})]$ yields a nontrivial upper triangular matrix with eigenvalues $y^2$ and $y^{-2}$ if $y\neq\{\pm 1\}$. In particular, $[\hat{x},\mathrm{diag}(y,y^{-1})]$ can have order $r$ for any odd prime $r$ dividing $(q-1)/2$. Since $\hat{x}=\hat{x}[\hat{x},\mathrm{diag}(y,y^{-1})]$, this gives us what we need.   The only other case is $G=\textrm{M}_{10}$ and $o(x)=8$. There, one can check directly using the character table that there exists $y\in G\setminus x^G$ of order $8$ in $\textrm{M}_{10}$ with $xy$ of odd order.
\end{proof}

Let $x\in G$ be a $p$-element. Consider the following property for the pair $(G,x)$:
$$(\PP_3): r\mid o(xy) \text{ for all nontrivial $r$-elements $y\in G$ and for all primes $r\neq p$}.$$  

It is easy to see that $\PP_3$ is a Baer-Suzuki property. Next, let $K=O_p(G)$. We show that $(G,x)$ satisfies $\PP_3$ if and only if $(G/K,Kx)$ satisfies $\PP_3.$

Assume that $(G,x)$ satisfies $\PP_3$. Let $xK\in G/K$ be an $r$-element for some prime $r\neq p.$ By Lemma \ref{l:prime powers}, we can assume that $y$ is an $r$-element. Then $xy$ is $r$-singular and since $r\nmid |K|$, we see that $Kxy=KxKy$ is also an $r$-singular element in $G/K.$  

Conversely, assume that $(G/K,Kx)$ satisfies $\PP_3.$ Let $y\in G$ be an $r$-element for some prime $r\neq p.$ Now $Kx$ is an $r$-element in $G/K$ and so $KxKy=Kxy$ is $r$-singular in $G/K$ which implies that $r\mid o(xy)$.
 
 \begin{proof}[\textbf{Proof of Theorem \ref{t:thmE}}] 
 Let $G$ be a counterexample to the theorem with $|G|$ minimal. Then $r\mid o(xy)$ for all nontrivial $r$-elements $y\in G$, where $r\neq p$ is a prime, but $x\not\in O_p(G)$. 
 It follows that $\langle x \rangle$ is weakly subnormal in $G$.  Moreover, we can assume that $O_p(G)=1$ and so Proposition \ref{p:reduction} applies.  If $O_{p'}(G)$ is not central,  then the coset $xO_{p'}(G)$ contains
 different conjugates of $x$ and so the result holds.  So $O_{p'}(G)$ is central, whence $F^*(G)=E(G)$.  
 
 Suppose first that $E(G)$ is not quasisimple, and let $r\neq p$ be a prime dividing the order of a component. Then by Lemma \ref{l:wreath2}, there exists $z\in E(G)$ such that $y:=[x,z]$ is a non-trivial $r$-element. Then $xy=x^z$ is not divisible by $r$.

Thus, $G/Z(E(G))$ is almost simple, and the possibilities are given in Proposition \ref{p:reduction}. 
Clearly, it will suffice to assume that $G$ is almost simple, and to find a prime $r\neq p$ such that $r$ does not divide the order of the Schur multiplier of $S:=E(G)/Z(E(G))$, and such that there exists an $r$-element $y$ of $G$ with $o(xy)$ not divisible by $r$.  
 
 If $S$ is sporadic, a straightforward computation using the character table proves the claim above.   If $S$ is alternating, then $S=\Alt_p$ with $p\ge 13$ and  $x$ is a $p$-cycle and so $[x,z]=y$ can be an element of order $3$. Hence,  $xy=x^z$ has order $p$ and is prime to $3$.  
 
Suppose next that $S$ is a simple group of Lie type and that $x$ is a regular semisimple element. Let $r\neq p$ be a prime divisor of $|S|$ not dividing the order of the Schur multiplier of $S$, and not equal to the defining characteristic of $S$. Let $z$ be an element of $S$ of order $r$.  By  Gow's  theorem \cite[Theorem 2]{Gow}, $z^g=xa$ for some $a\in x^G$, $g\in G$. Then $o(xz^{-g})=o(z^{-g}x)=o(a)=o(x)$.

If $G=\LL_2(q)$, $p=2$ and $q=p=2^k+1$, then the argument above also gives the result. Indeed, in those cases, $q+1$ has at least two distinct odd prime divisors $r_1,r_2$. Then \cite[Theorem 2]{Gow} yields $x=x_1x_2$ where $x_i$ is a regular semisimple $r_i$-elements of $G$. Then $o(xx_2^{-1})=o(x_1)$. If $G=\PGL_2(q)$, $p=2$ and $q\equiv 3 \pmod{4}$, then as in the proof of Theorems \ref{t:thmA} and \ref{t:thmC}, there exists $z$ in the split torus of odd prime order such that $y=[x,z]$ has order $o(z)$. Then $xy=x^z$. 

If $p$ divides $q$ we are in the case $G=\LL_2(p), p > 5$. Then
$$x:=\begin{pmatrix}
 1 & 1\\
 0 & 1
\end{pmatrix}, 
y:=\begin{pmatrix}
 0 & -1\\
 1 & 1    
\end{pmatrix},
z:=\begin{pmatrix}
  1 & 0\\
 1 & 1   
\end{pmatrix}
$$
yields $o(x)=o(z)=p$ and $o(y)=3$.

The cases with $p=5$ and $G={^2}{\rm B}_2(32).5$ and 
$p=3$ with $G=\LL_2(8).3$ with $x$ an outer element of order $9$ are straightforward to check.   Similarly, 
the results for 
$x$ an outer element of order $8$ in $\textrm{M}_{10}$ can be ruled out by using \cite{GAP}.  This completes the proof.
 \end{proof}

\section{Glauberman's $Z_p^*$-theorem }\label{s:Glauberman}

In the next theorem, we collect and prove several known equivalent statements of Glauberman's $ Z_p^*$-theorem including the proof of Theorem \ref{t:thmB}. (See \cite{Ar, Glauberman, GGLN,GR,Tong}.) Recall that  for a finite group $G$ and a prime $p,$ $ Z^*_{p}(G)/ O_{p'}(G)= Z(G/ O_{p'}(G))$. Moreover, for a $p$-element $x\in G$  and a subgroup $P$ of $G$ containing $x$, we say that $x$ is   {isolated} (or strongly closed) in  $P$ with respect to $G$ if $x^G\cap P=\{x\}$, that is, $x$ is not conjugate in $G$ to any element in $P-\{x\}$.

\begin{thm}\label{th:Glauberman} Let $G$ be a group and let $p$ be a prime. Let $x\in G$ be a $p$-element and let $P$ be a Sylow $p$-subgroup of $G$ containing $x$. Then the following are equivalent.

\begin{enumerate}[\rm (i)]

\item $x$ is isolated in $P$ with respect to $G$, i.e., $x^G\cap P=\{x\}$. 
\item $x^G\cap  C_G(x)=\{x\}$, that is, $x$ does not commute with any $G$-conjugate of $x$ different from $x$.

\item $ C_G(x)$ controls $p$-fusion in $G$, that is, $ C_G(x)$ contains a Sylow $p$-subgroup $P_1$ of $G$ and if $y,y^g\in P_1$ for some $g\in G$, then $y^g=y^h$ for some $h\in  C_G(x)$.

\item $[x,g]$ is a $p'$-element for all $g\in G.$
\item $[x,g]$ is a $p'$-element for all elements $g\in G$ of prime power order.

\item $x\in Z_p^*(G)$,  that is, $x$ is central modulo $ O_{p'}(G)$.
\item $G= C_G(x) O_{p'}(G)$.
\end{enumerate}
\end{thm}

\begin{proof}  Let $x\in G$ be a $p$-element. Let $P\in\Syl_p(G)$ with $x\in P$, $C= C_G(x)$ and $X=\langle x\rangle.$

$(i)\Leftrightarrow (ii)$. Assume that $x^G\cap P=\{x\}$. It follows that $x^P\subseteq x^G\cap P=\{x\}$ and hence $x\in Z(P)$. Thus $P\subseteq C$ and so $P$ is a Sylow $p$-subgroup of $C$. Clearly $x\in x^G\cap C$. Now let $g\in G$ be such that $x^g\in C$. It follows that $\langle x,x^g\rangle$ is a $p$-subgroup of $C$. By Sylow's theorem, $\langle x,x^g\rangle\leq P^h$ for some $h\in C.$ We now have that  $x^{gh^{-1}}\in P$ and $x^h=x$. So $x^{gh^{-1}}\in x^G\cap P=\{x\}$ which forces $x^g=x^h=x$ proving (ii).

Assume that $x^G\cap C=\{x\}$. Let $U$ be a Sylow $p$-subgroup of $C$ containing $x$. We claim that $U$ is also a Sylow $p$-subgroup of $G$. Assume by contradiction that $U$ is not a Sylow $p$-subgroup of $G$ and suppose that $U\leq P_1\in\Syl_p(G)$. By Sylow's theorem, $P_1=P^t$ for some $t\in G.$ Since $|U|<|P_1|$, $U_1:= N_{P_1}(U)>U$. Let $g\in U_1$. Then $U^g=U$ which implies that $x^g\in U\cap x^G\subseteq x^G\cap C=\{x\}$. Hence $x^g=x$ and so $g\in C$. Therefore $U_1\subseteq C$ which is impossible as $U$ is a Sylow $p$-subgroup of $C$ and $U_1$ is a $p$-group properly containing $U$.

$(i)\Leftrightarrow (iii)$. This is \cite[Lemma 2.3]{Tong}. 

Assume that $x^G\cap P=\{x\}$. We claim that $C$ controls $p$-fusion in $G$. Since $x^P\subseteq x^G\cap P,$ $x\in Z(P)$ and so $P\leq C$.  Now assume $y,y^g\in P$ for some $g\in G.$ Since $y,y^g\in P\subseteq  C_G(x),$ $\{x,x^{g^{-1}}\}\subseteq  C_G(y)$. Let $U$ be a Sylow $p$-subgroup of $ C_G(y)$ containing $x$. By Sylow's theorem, $U\leq P^t$ for some $t\in G.$ It follows that $x^{t^{-1}}\in P\cap x^G=\{x\}$; hence $x^{t^{-1}}=x$ and so $t\in  C_G(x)$. Now $x^{g^{-1}}\in U^c$ for some $c\in C_G(y)$ as $U$ is a Sylow $p$-subgroup of $C_G(y)$. Now we have $x^{g^{-1}c^{-1}t^{-1}}\in P\cap x^G=\{x\}$ which implies that $g^{-1}c^{-1}\in C.$ Therefore $h=cg\in C$ and so $y^g=y^{cg}=y^h$ where $h\in C$. Thus $C$ controls $p$-fusion in $G$ as wanted.

Conversely, assume that $C$ controls $p$-fusion in $G$ and let $P_1$ be a Sylow $p$-subgroup of $C$. By definition, $P_1\in\Syl_p(G)$ and thus $P_1^t=P$ for some $t\in G.$ Since $x\in P=P_1^t,$ $x^{t^{-1}}\in P_1\leq C$. As $C$ controls $G$-fusion in $P_1$, it follows that $x^{t^{-1}}=x^h$ for some $h\in C$. Hence $x^{t^{-1}}=x$ and so $t\in C.$ Thus $P=P_1^t\subseteq C$. Now if $x^g\in P$ for some $g\in G$, then $x^g=x^u$ for some $u\in C$ and so $x^g=x$. We conclude that $x^G\cap P=\{x\}.$

$(vi)\Leftrightarrow (vii)$. Clearly $(vii)$ implies $(vi)$. We will show the other direction. Assume that $\overline{x}\in Z(\overline{G})$, where $\overline{G}:=G/ O_{p'}(G)$. Let $X=\langle x\rangle$. Then $X$ is a $p$-subgroup of $G$ and $\overline{X}$ is a central subgroup of $\overline{G}$. By \cite[Lemma 7.7]{Isaacs}, we have $ C_{\overline{G}}(\overline{X})=\overline{ C_G(X)}=\overline{C}$, hence $\overline{G}=\overline{C}$ or $G=C O_{p'}(G)$. This proves the remaining implication.

$(vii)\Rightarrow (iv)$. Assume that $G=C O_{p'}(G)$. Then $G= O_{p'}(G)C$. Let $g\in G.$ Then $g=tc$ for some $c\in C$ and $t\in O_{p'}(G)$. Now $[x,g]=[x,tc]=[x,c][x,t]^c=[x,t]^c.$ As $t\in O_{p'}(G)\unlhd G$, we see that $[x,t]=(t^x)^{-1}t\in  O_{p'}(G)$ and hence $[x,g]=[x,t]^c\in  O_{p'}(G)$ is a $p'$-element. This proves $(v)$.

$(iv)\Rightarrow (v)$. This is obvious.

$(v)\Rightarrow (i)$.
(The next two claims prove Theorem \ref{t:thmB}) 
Let $Y$ be a $p$-subgroup of $G$ containing $x$. Then $x\in Y\unlhd  N_G(Y)$ and thus  $[x,g]\in Y$ is a $p'$-element for every prime power order element $g\in N_G(Y)$, it follows that $[x,g]=1$  and so $x$ centralizes every prime power order element of $ N_G(Y)$. Hence $x$ centralizes $ N_G(Y)$ so $ N_G(Y)\leq  C_G(x)$. In particular, $ N_G(X)= C_G(X)$ and $ N_G(P)\leq  C_G(X)$, where $X=\langle x\rangle\leq P\in\Syl_p(G) $.

Assume that $x^g\in P$ for some $g\in G.$ Then $x\in P^{g^{-1}}$ and so $P^{g^{-1}}\leq  C_G(x)$ by the previous claim. Since $P,P^{g^{-1}}\leq  C_G(x)$, $P^{g^{-1}}=P^u$ for some $u\in  C_G(x)$, hence $P^{ug}=P$ which implies that $ug\in N_G(P)\leq  C_G(x)$. It follows that $ug\in C_G(x)$, therefore $g\in C_G(x)$. We have shown that if $x^g\in P$, then $x^g=x$ for any $g\in G$. Therefore, $x^G\cap P=\{x\}$ and so $x$ is isolated in $P$ with respect to $G$.

$(i)\Rightarrow (vi)$. Assume $x^G\cap P=\{x\}$.  Since $(i), (ii)$ and $(iii)$ are equivalent, 
we also have that  $x^G\cap C=\{x\}$. In particular, $x\in Z(P)$ and $P\in\Syl_p(C)$. Assume $o(x)=p^a$ for some integer $a\ge0.$ 

By \cite[Lemma 3.2]{GGLN} or \cite[Lemma 2.5]{Tong}, if $y\in \langle x\rangle,$ then $y^G\cap P=\{y\}$.  If $a=0$, then there is nothing to prove. Assume $a\ge 1.$ Let $y=x^{p^{a-1}}$. Then $o(y)=p$ and $y^G\cap P=\{y\}$ or equivalently $y$ does not commute with any conjugate $y^g\neq y.$ By \cite[Theorem 4.1]{GR}, $y$ is central modulo $N:= O_{p'}(G)$.

Let $\overline{G}=G/N$. Then $\overline{x}$ is isolated in $\overline{P}$ with respect to $\overline{G}$. As $ O_{p'}(\overline{G})=1$, if $N$ is nontrivial, then $\overline{x}$ is central in $\overline{G}$ by induction, which proves $(vi)$. Thus we can assume that $N=1.$ It follows that $Z=\langle y\rangle \subseteq  Z(G)$. Again, $xZ$ is isolated in $P/Z$ with respect to $G/Z$. By induction, $xZ$ is central modulo $K/Z= O_{p'}(G/Z)$. Since $Z$ is a central $p$-subgroup of $K$ with $K/Z$ a $p'$-group, $K$ is $p$-solvable with a central Sylow $p$-subgroup $Z$. By Hall's Theorem \cite[Theorem 3.20]{Isaacs}, $K$ has a Hall $p'$-subgroup $H$ and $K=HZ$. Since $[Z,K]=1$, $H\unlhd K$ and so $H\unlhd G$. Since $ O_{p'}(G)=1$, we deduce that $H=1$ and hence $ O_{p'}(G/Z)=1$. Thus $xZ\in Z(G/Z)$ and hence $[x,g]\in Z\subseteq  Z(G)$ for all $g\in G$. It follows that $x^g$ commutes with $x$ for all $g\in G$ which forces $x^g=x$ for all $g\in G$ (since $x$ is $P$-isolated). Hence $x\in Z(G)$ as wanted.
\end{proof}

\section{Orders of commutators and the open conjectures} \label{s:conjectures}

We prove Conjecture \ref{conjD} under the assumption that $  O_p(G)$ is abelian.    The structure of the argument
is different in this case because this is not a good inductive hypothesis.  So Wielandt's  Zipper lemma is not as useful
in this context.   However, we can make a number of reductions of a similar nature.

We first need a classification of subgroups of prime order satisfying a variation of the weakly subnormal property.

\begin{lem}\label{l:almostwsn}  Let $p$ be a prime, $G$ a finite group with $O_p(G)=1$ and $x \in G$ of order $p$.  Assume
that $G = \langle x^g | g \in G \rangle$ and that if $x \in H$ a proper subgroup of $G$, then $O_p(H) \ne 1$.
Then one of the following holds:
\begin{enumerate} 
\item[\rm(i)]  $\langle x \rangle$ is weakly subnormal in $G$ and a Sylow $p$-subgroup of $G$ is cyclic; or 
\item[\rm(ii)]  $F^*(G) \cong \LL_p^{\epsilon}(2^a)$ with $2^a - \epsilon = p$ a Fermat or Mersenne prime or
$F^*(G) \cong \rm{U}_3(8)$ with $p=3$.  
\end{enumerate}
\end{lem} 

\begin{proof}  If $p=2$, the result follows since $x$ must be contained in a dihedral group of order $2r$ for some odd prime $r$, by the Baer-Suzuki theorem. So assume that $p$ is odd. 

Suppose that $F(G)$ is noncentral.  Then $x$ acts nontrivially on some $Q:=O_r(G)$ for $r \ne p$
and so $G = \langle O_r(G), x \rangle$.  Moreover, $x$ must act irreducibly on $O_r(G)/\Phi(O_r(G))$ and centralize $\Phi(O_r(G))$
whence $\langle x \rangle$ is weakly subnormal in $G$.  It follows that $F(G)$ is central and indeed is contained in the Frattini subgroup
of $G$ (otherwise $G$ contains a supplement to  $F(G)$ which contradicts the fact that $G$ is the normal closure of $\langle x \rangle$).

Thus, $G=\langle E(G), x \rangle$ and $x$ acts transitively on the components of $E(G)$.   If there is more than $1$ component, 
$x$ will normalize a Sylow $r$-subgroup of $E(G)$ for any $r \ne p$,  a contradiction.

So $S:=E(G)$ is quasisimple.    If $S/Z(S)$ is sporadic, this it is a straightforward computation to check that (i) holds.  If $S/Z(S)$ is an alternating
group and $n > p \ge 5$, $x$ is in a Young subgroup that is a product of two nonabelian simple groups, contrary to assumption. So it reduces to the case
$n=p$ where the result is clear. If $p=3$, it reduces to the cases of $A_5$ and $A_6$.  In those cases an element of order $3$ is contained in a subgroup isomorphic to $A_4$.

So assume that $S/Z(S)$ is a finite simple group of Lie type in characteristic $r$.  If $r=p$ and $x \in S$ is unipotent, then
$x$ is in some subgroup $K \cong \SL_2(p)$ or $\PSL_2(p)$ \cite{Testerman2, Testerman} unless possibly  $p=3$ 
 (recall $p \ne 2$) and 
 $S={\rm G}_2(q)$ or ${^2}{\rm G}_2(3^a)$.    If $S=\PSL_2(p)$ or $\SL_2(p)$, then $p > 3$, $\langle x \rangle$ is weakly subnormal, and Sylow $p$-subgroups of $S$ are cyclic.
 
In the case of ${^2}{\rm G}_2(3^a)$,  there are two conjugacy classes of subgroups of order $3$.
One is contained in a $\PSL_2(3^a)$ and the other normalizes but does not centralize a maximal torus,
whence the result holds.  
If $G={\rm G}_2(q)$, then any class of elements of order $3$ other than  the class $(\tilde{A_1})_3$
is contained in  an $A_1$ subgroup.  If $x \in (\tilde{A_1})_3$,  then $x$  is conjugate to an element of ${\rm G}_2(3)$. 
  
So we may assume that either $x\not\in S$, or $r\neq p$. Suppose first that $x$ is an inner diagonal automorphism of $S$, so that $r\neq p$. Then $x$ cannot normalize any parabolic subgroup (because
 then $x$ normalizes but does not centralize its unipotent radical).  So $x$ is a regular semisimple element.
 
If $x$ lifts to an element $\hat{x}$ of order $p$ in the Schur cover $\hat{S}$ of $S$, then the Sylow $p$-subgroup of $S$ is cyclic. The only overgroups of $x$ with nontrivial $p$-core are contained in the normalizer of $\langle x \rangle$
and so $x$ is weakly subnormal.   
 
If $x$ lifts to an element $\hat{x}$ of order at least $p^2$ in $\hat{S}$, then $p$ must divide the order of the center of $\hat{S}$ and $\hat{x}^p$ is central (and must be trivial in $G$). The only possibilities are $S/Z(S) = \LL_p^{\epsilon}(q)$; or $p=3$ and $S=E_6^{\epsilon}(q)$. It is clear that the latter case does not occur (an element of order $9$ is not
 regular semisimple in $E_6^{\epsilon}(q)$). In the former case, $x$ will normalize a diagonal torus,
 so our hypothesis implies that $q-\epsilon$ is a power of $p$. Thus, either $(q,\epsilon,p)=(8,-1,3)$, or $q=2^a$ with $p=2^a-\epsilon$ a Fermat or Mersenne prime. Hence, (ii)  holds.   
   
Suppose next that either $x$ is a field automorphism; or that $p=3$, $S={^3}{\rm D}_4(q)$, and $x$ is a graph automorphism. Then $x$ normalizes a Borel subgroup and so acts nontrivially on a Sylow $r$-subgroup of $S$. It follows from our assumption that $r = p$. Then $x$ acts nontrivially on some maximal torus and so the result holds.   
 
The remaining case is $p=3$
and $x$ induces a graph or graph-field automorphism of $S/Z(S)={\rm D}_4(q)$.  If $x$ is a graph or graph-field automorphism of order $3$,  then $x$ acts nontrivially
on a long root subgroup and the result follows unless $r=3$. So assume $r=3$. If $x$ is a graph-field automorphism,
$x$ centralizes a  torus $T$ contained in $C_S(x) = {^3}{\rm D}_4(q)$ and acts nontrivially on $C_S(T)$
(which has trivial $3$-core and so the result follows). If $x$ is a graph automorphism of order
$3$ and $q$ is not a power of $3$,  then $x$ normalizes but does not centralize a long root subgroup.
If $q$ is a power of $3$,  $x$ normalizes ${\rm D}_4(3)$. 
 
\end{proof}

We next prove a strong version of Conjecture \ref{conjD} under the assumption that $O_p(G)$ is abelian.

\begin{thm}\label{t:conjDOp}
Let $G$ be a finite group and let $p$ be a prime.   Assume  
that $O_p(G)$ is abelian.
Let $x\in G$ be an element of order $p$ not contained in $O_p(G)$.  
 Then there exists  $g \in G$ such that $y=[x,g]$ is a nontrivial $p'$-element. 
\end{thm}
 \begin{proof} Let $(G,x)$ be a counterexample to the theorem  with $|G|$ minimal, and set $V:=O_p(G)$. By the Baer-Suzuki theorem, $p > 2$. Also, by the minimality of $(G,x)$ as a counterexample, we have $O_{p'}(G)=1$.

 If $x$ is contained in a proper subgroup $H$ of $G$
  with $O_p(H)\le V$, then the result follows by induction (since $O_p(H)$ is still abelian).   So
  we may assume this is not the case. In particular, $G=\langle x^g \mid g \in G \rangle$ and $O_p(H/V)\neq 1$ for all proper overgroups $H$ of $x$ containing $V$, so the previous lemma applies (in $G/V$). 

We first show that there exists a $g \in G$ such that $[x,g]=y$ reduces to a $p'$-element
in $G/V$ and moreover,  $xV$ and $yV$ invariably generate $G/V$.   First assume that 
$\langle x \rangle$ is weakly subnormal in $G$.  Then we just need to choose $g$ so that
$yV$ is a $p'$-element not conjugate to an element of $M$, the unique maximal subgroup of $G$ containing $x$.

If $G$ is $p$-solvable, then this is clear. Indeed, writing $G/V=Q\langle xV\rangle$ as in Theorem \ref{t:psolvable}, we have that $xV$ and any element of $Q \setminus Z(Q)$ generates $G/V$.   
So assume that $G$ is not $p$-solvable and so by the (proof of the) previous lemma, one
of the following holds:
\begin{enumerate}
\item  $G/V$ is a sporadic simple group;
\item  $G/V \cong A_p, p \ge 5$; 
\item $G/V \cong \SL_2(p)$ or $\PSL_2(p)$;  or
\item $G/V$ is a quasisimple group of Lie type,  $x$ is a regular semisimple element
not contained in any parabolic subgroup and the Sylow $p$-subgroup of $G$ is cyclic.
\end{enumerate}

The first case is an easy computation in MAGMA.   For alternating groups, we
choose $g$ so that $y$ is $3$-cycle which does not normalize an element of order $p$.
In the case of $\SL_2(p)$ or $\LL_2(p)$, a straightforward computation shows that we can
choose $g$ so that $[x,g]$ is an element of a nonsplit torus (and so is not in a Borel subgroup).

In the fourth case since $x$ is a regular semisimple,   given any semisimple element $y$
we can choose $g$ so that $[x,g]=y$ by Gow's result \cite{Gow}.  In particular, we can choose $y$ to have order prime to $p$
and not contained in $N_G(\langle x \rangle)$ (for example choose $y$ to be regular semisimple 
in some maximal torus that has order prime to $p$).   

Suppose finally that case (ii) from Lemma \ref{l:almostwsn} holds.  In this case $x$ is contained in exactly two maximal subgroups (the normalizer of a quasi-split torus and the normalizer of an irreducible torus).  In particular,
$x$ is regular semisimple and by a slight extension of the result of Gow, we can choose
$g$ with $y=[x,g]$ any noncentral semisimple element in the derived subgroup.  Again, we
choose $y$ to be a regular semisimple of order prime to $p$  in a maximal torus that has order
prime to $p$ and is neither in the quasi-split torus nor the irreducible torus.

So we have shown in all cases, that we can choose $g \in G$ such that $[x,g]=yv$
where $y$ is a nontrivial $p'$-element and $v \in O_p(G)$.   Moreover, 
$x$ and $y$ invariably generate $G$ (all we require is that they invariably generate $G$ modulo $O_p(G)$).

Next, let $W:=[V,G]$. We claim that $[V,G]=[V,G,G]$, and that the element $v$ above can be taken to be an element of $W$. To see this, note first that $V/W$ is the trivial module. If $G$ is $p$-solvable, then since $G=(V\rtimes Q)\rtimes\langle x\rangle$, we can argue as in the proof of Proposition \ref{p:wsnprime} to see that $C_Q(V)=1$. Thus, $V=[V,Q]=[V,G]$, so $V=[V,G]=[V,G,G]$, which proves the claim.

Assume now that $G$ is not $p$-solvable. If $\langle x\rangle$ is weakly subnormal, then one of the cases (1)--(4) above holds. Using Theorem \ref{t:cyclic} for case (1), we see that the $p$-part of the Schur multiplier is trivial in each case. If $W\neq V$, then it follows that either $W=V$, giving us what we need, or $G/W\cong A\times J$, where $A$ is an abelian $p$-group, and $J$ is as in (1)--(4). In the latter case, since $G=\langle x^g\mid g\in G\rangle$, and $G$ can be invariably generated by $x$ and $yv=[x,g]$ with $y$ a $p'$-element, we have $A\cong C_p$. It follows that the Schur multiplier of $A\times J$ also has trivial $p$-part, whence the same argument shows that $G/[W,G]\cong G/W$, i.e. $W=[W,G]$. Since $yv=[x,g]$, we have $v\in W$, as needed. 

Suppose finally that case (ii) in Lemma \ref{l:almostwsn} holds. Then $J:=G/V$ is almost simple with socle either $\LL_p^{\epsilon}(2^a)$ with $2^a-\epsilon=p$ a Fermat or Mersenne prime; or $\mathrm{U}_3(8)$ with $p=3$. Arguing as above, we see that the only possibilities are $W=[W,G]$ or $G/[W,G]\cong A\times \hat{J}$, where $\hat{J}$ is a Schur cover of $J$ with cyclic centre of order divisible by $p$, and $A\cong C_p$. Suppose that the latter case holds. Then $V/[W,G]=A\times Z(\hat{J})\le Z(G/[W,G])$, which implies $[V,G]=[W,G]$, i.e. $W=[W,G]$. Again, since $yv=[x,g]$, we must have $v\in W$, whence the claim.

Now, the fact that $[W,G]=W$ implies that $G$ has no nontrivial fixed points on $W^*$, the character group of $W$. Since $x$ invariably generates $G$ with any element from $yW$,
this implies that for any nontrivial linear character $\phi$ of $W$,  the stabilizer of $\phi$ cannot contain a
conjugate of both $x$ and an element of $yW$.   Thus, for any irreducible character $\chi$ of
$G$ that is nontrivial on $W$, we have $\chi(x)\chi(yw)=0$ for all $w \in W$.  

Let $N$ be the number of ways of writing an element of $y^G$ as  product of conjugates of $x$
and $x^{-1}$.   Up to a constant, this is
$$
\sum_{\chi}  \frac{|\chi(x)|^2\chi(y)} {\chi(1)},
$$
where the sum is over all irreducible characters of $G$.
By the above remarks, it suffices to only consider characters of $G/W$ and in particular,
this number is the same for all $yw$ and in particular,  $[x,g]$ is a $p'$-element for some $g \in G$.
\end{proof}

In particular, we have the following:
 
 \begin{cor}\label{c:conjDOp}
Let $G$ be a finite group and let $p$ be a prime. 
Let $x\in G$ be a nontrivial element of order $p$ and assume that a Sylow $p$-subgroup of $G$ is abelian.
 Then $[x,g]$ is a nontrivial $p'$-element for some $g \in G$.
  \end{cor}

Let $G$ be a finite group and let $p$ be a prime. In \cite[Theorem 2.1]{GR}, the authors proved that if $x$ is a $p$-element for some prime $p$ and assume that $[x,g]=1$ or $[x,g]$ is $p$-singular for every $g\in G$, then $x\in  O_p(G)$ provided that $G$ has a cyclic Sylow $p$-subgroup. The authors then ask whether this could be true without the restriction on the Sylow $p$-subgroups. By the Baer-Suzuki theorem, this question has a positive answer if $x$ is an involution. To see this, assume that $x\in G$ is an involution and that $[x,g]=1$ or $2$-singular for every $g\in G.$ Since  $[x,g]=xx^g$, if we can show that $[x,g]=xx^g$ is a $2$-element for every $g\in G$, then $\langle x,x^g\rangle$ is a $2$-group for every $g\in G$ and thus by Baer-Suzuki theorem, $x\in  O_2(G)$.  Assume that this is not the case and let $g\in G$ be such that $z:=[x,g]=xx^g$ is not a $2$-element. Then $o(z)=2^am$, where  $a,m$ are integers with $m>1$ being odd. Note that $z^x=z^{-1}$. Let $y=z^{2^a}$. Then $o(y)=m$ is odd and $y^x=y^{-1}$. So $[x,y]=x^{-1}y^{-1}xy=(y^x)^{-1}y=y^2$ is $2$-regular, so $y^2=1$ which forces $y=1$, a contradiction.   In Theorem \ref{t:abeliansylow}, we generalized to the case
of abelian Sylow $p$-subgroups. 

However, this question turns out to be false for other nontrivial $2$-elements in general. The group $\GL_2(3)$ has an element $x$ of order $8$ such that $o([x,g])\in\{1,4,6\}$ for all $g\in\GL_2(3)$ but $x\not\in  O_2(\GL_2(3))$.  

 The following is a structure result for groups with a weakly subnormal subgroup of prime order.

\begin{prop} \label{p:wsnprime}  Let $G$ be a finite group with a weakly subnormal subgroup $R=\langle x\rangle$ of prime order $p$. Then either
\begin{enumerate}[\rm (i)]
\item $p=2$ and $G\cong D_{2q}$ with $q$ an odd prime; or

\item $p$ is odd and one of the following holds.
\begin{enumerate}
\item[\rm(a)] $O_{p'}(G)$ is non-central and $G = QR$ with $Q=O_{p'}(G)$ a special $q$-group, and $R$ acting faithfully and irreducibly on $Q/\Phi(Q)$.

\item[\rm(b)]  $O_{p'}(G) \le Z(G) \cap \Phi(G)$, and $G/\Phi(G)=V\rtimes L$, where $V:=O_p(G)\Phi(G)/\Phi(G)$ and $L$ is a completely reducible subgroup of $\GL(V)$ of shape $L=Q\rtimes\langle y\rangle$, with $Q$ a nonabelian special $q$-group, $q\neq p$, and $|y|=p$. Further, $x=vy$ for some $v\in V$. 
\item[\rm(c)] $G/O_p(G)$ is quasisimple.
\end{enumerate}
\end{enumerate}
\end{prop}
\begin{proof}  
If $p=2$ then it is clear that $G\cong D_{2q}$ with $q$ an odd prime, so we will assume for the remainder of the proof that $p$ is odd. 

If $O_{p'}(G)$ is non-central then the result follows from Theorem \ref{t:psolvable}. So assume that $O_{p'}(G)\le Z(G)\cap \Phi(G)$. Then $G/O_p(G)$ is as in (a), so we just need to prove the structure result on $G/\Phi(G)$. Thus, we may assume that $\Phi(G)=1$. Then $O_{p'}(G)=1$, so $F(G)=O_p(G)$ and $G$ embeds as a subdirect subgroup of a group $X:=V_1:L_1\times\hdots\times V_s:L_s$ containing $V_1\times\hdots\times V_s$, where each $V_i$ is an elementary abelian $p$-group, and $L_i\le \GL(V_i)$ is irreducible. In particular, $G$ acts completely reducibly on $F(G)/\Phi(G)=O_p(G)\Phi(G)/\Phi(G)$. 

Now write $x=vy$, with $v=v_1+\hdots+v_s$, $v_i\in V$. If $v_i\in [V_i,x]$ for any $i$, then by replacing $x$ by a $V$-conjugate, we could assume that $v_i=0$. But then $x$ is contained in the maximal subgroup $\langle \hat{V}_i,Q,y\rangle$ of $G$, where $\hat{V}_i=\sum_{j\neq i}V_j$. This is a contradiction, so we have $v_i\not\in [V_i,x]$ for any $i$. 

All that remains is to prove that $Q$ is not elementary abelian. So assume that $Q$ is elementary abelian.
Note first that $Q$ has no fixed vectors in $V$. Indeed, otherwise, $[QV,V]=[Q,V]$ would be a proper $G$-normal subgroup of $V$ contained in $V$. But then $G/[Q,V]\cong ((V/[Q,V])\times Q)\rtimes\langle y\rangle$. It follows that $G$ has a quotient isomorphic to $V/[Q,V]\rtimes \langle y\rangle$, whence has an elementary abelian $p$-quotient of order at least $p^2$. This contradicts $G=\langle x\rangle^G$.

So $Q$ has no fixed vectors on $V$. Since $x$ acts irreducibly on $Q$, it follows that $L$ acts faithfully on each of the groups $V_i$. If $Q$ acts homogeneously on $V_1$, then $L$ is quasiprimitive on $V_1$, since $Q$ is the only non-trivial proper normal subgroup of $G$. By the structure theory of quasiprimitive groups, this would imply that $Q$ is cyclic of order $q$, $n:=\dim{V_1}$ is divisible by $p$, and $L$ lies in $Z(\GL_{n}(p^n)).p$. But then $y$ acts on $V_1=\mathbb{F}_{p^n}$ via $\mu\rightarrow \mu^{n/p}$, for $\mu\in \mathbb{F}_{p^n}$. It follows from an easy field calculation that since $x=vy$ has order $p$, we have $v_1\in [V_1,x]$ --  a contradiction. 

So we must have that $Q$ is non-homogeneous. Since $o(x)=p$, it follows that $V_1$ is a direct sum of permutation modules for $\langle x\rangle$, whence $V_1$ acts transitively by conjugation on the coset $V_1x$. In particular, $x=v_1y$, so by replacing $x$ by a $V_1$-conjugate, we may assume that $v_1=0$. Arguing as in the paragraph above then gives the required contradiction.
\end{proof}

We can now describe the structure of the minimal counterexamples to Conjecture \ref{conjD}.
\begin{cor}\label{p:minimal2}  Let the pair $(G,x)$ be a counterexample to Conjecture \ref{conjD} with $|G|$ minimal. Then $p:=o(x)$ is odd, $\langle x\rangle$ is weakly subnormal in $G$ and there is a unique maximal subgroup $M$ of $G$ containing $x$ with $x\in  O_p(M)$ but $x\not\in  O_p(G)$. Moreover, $G=\langle x^G\rangle$, $O_{p'}(G)=1$, 
and one of the following holds.
\begin{enumerate}[\rm (i)]
\item $G=PQ$,  $P=\langle x\rangle  O_p(G)\in\Syl_p(G)$, $Q  O_p(G)=  O_{p,q}(G)$ for some prime $q\neq p$, $Q$ is a nonabelian special $q$-group, $M=PR_0$ with $P=  O_p(M)$ and  $R_0\leq   Z(Q)$, and $\overline{Q}/  Z(\overline{Q})$ is a faithful irreducible $\mathbb{F}_q\langle \overline{x}\rangle$-module, where $\overline{G}=G/O_p(G)$. In particular, $G$ is solvable.
\item $G/O_p(G)$ is quasisimple.
\end{enumerate}
\end{cor}
\begin{proof} Clearly $R:=\langle x \rangle$ is a weakly subnormal $p$-subgroup (but we know nothing about $O_p(G)$).   If $O_{p'}(G)$ is not central, then $[x,g]$ is a nontrivial
$p'$-element for any $g \in O_{p'}(G)$ not centralizing $x$. Thus, $O_{p'}(G)\le Z(G)$, and we can pass to $G/O_{p'}(G)$. The minimality of $|G|$ therefore implies that $O_{p'}(G)=1$.    

Assume first that $G$ is $p$-solvable. Then the previous proposition applies and we see that $G=PQ$,  $P=\langle x\rangle  O_p(G)\in\Syl_p(G)$, $Q O_p(G)=  O_{p,q}(G)$ for some prime $q\neq p$, and $Q$ is a special $q$-group. Also, $\overline{G}:=G/O_p(G)$ acts faithfully and compltely reducibly on $O_p(G)/\Phi(G)$; and $\overline{x}$ acts faithfully and irreducibly on $\overline{Q}/  Z(\overline{Q})$.

Thus, all that remains is to show that $M=PR_0$, where $P=  O_p(M)$ and $R_0\leq   Z(Q)$. To see this, note that $O_p(M)$ contains $O_p(G)$, and $\overline{M}=\overline{R}\times \Phi(\overline{G})$. Also, by the previous proposition, either $\overline{Q}$ is elementary abelian and $\Phi(\overline{G})=\Phi(\overline{Q})=1$ or $\overline{Q}$ is special and $\Phi(\overline{G})=\Phi(\overline{Q})=[\overline{Q},\overline{Q}]=Z(\overline{Q})$. It follows that $O_p(M)=O_p(G)R=P$, and $\Phi(\overline{G})=\overline{R_0}$ for some $R_0\le Z(Q)$. The result follows.

So now assume that $G$ is not $p$-solvable.  
Since $p$ is odd and $o(x)=p$, it follows by Theorem \ref{t:cyclic} that 
$G/O_p(G)$ is quasisimple.  
\end{proof}

\vs

Now consider Conjecture \ref{conj: multicommutators}.  Fix a prime $p$   
and consider a minimal counterexample.  Then $\langle x \rangle$ is a weakly subnormal $p$-subgroup of $G$.    If $O_{p'}(G)$ is not central, then
the result is clear.   If $O_{p'}(G)\le Z(G)$, we can pass to the quotient and so $O_{p'}(G)=1$.   Note that if $G$ is a counterexample, then $G/O_p(G)$ is as well
and so $O_p(G)=1$. It follows from Corollary \ref{c:cyclic34} that $|x|\neq 2,4$ (although the case $|x|=2$ can already be dealt with using the Baer-Suzuki theorem). Further, Theorem \ref{t:cyclic} applies. We can rule out the small cases from Theorem \ref{t:cyclic} using GAP.

We note finally that $|x|\neq 3$. Indeed, if $|x|=3$, then Corollary \ref{c:cyclic34} implies that $G=\LL_2(2^e)$ for $e$ an odd prime. If $e=3$, then we verify directly that $\Gamma_k(x)=\Gamma_{k+1}(x)$ consists of $27$ elements of orders $9$, together with the identity element, for all $k\geq 2$. One can then check that there exists elements $g,h\in\Gamma_k(x)$ such that $gh$ is not a $p3$-element. Thus, we have $e > 3$.  We then observe that a Sylow $3$-subgroup of $G$ has order $3$, and so if $[y,x]=z$ with $y$ and $z$ $3$-elements, we see that $x^{-y}xz^{-1}=1$.
Hence, $\langle x,x^y,z\rangle$ is a $(3,3,3)$-group,  i.e. a group generated by two elements
of order $3$ whose product has order $3$. By \cite{Magnus}, such a group has an abelian normal subgroup of index $3$. 
In particular, since a Sylow $3$-subgroup of $G$ has order $3$, the commutator of two elements of order $3$ is a $3'$-group. 
Thus, $\Gamma_k(x)$, for $k > 1$, cannot contain elements of order $3$.
 
We have therefore proved the following:
\begin{prop} \label{p:conj1}   Let $p$ be a prime, and suppose that $(G,x)$ is a minimal counterexample to Conjecture \ref{conj: multicommutators}.  Then $|x|\not\in\{2,3,4\}$, and one of
the following holds:
\begin{enumerate}
\item[\rm(i)] $p \ne 2$, $G$ is a non-sporadic simple group,  a Sylow $p$-subgroup of $G$ is cyclic, and $G$ is given in Table 1; or
\item[\rm(ii)]  $p=2$,  $G=\LL_2(q)$ or $\PGL_2(q)$, $M$ is the normalizer of a nonsplit torus, $q$ is prime, $q \equiv -1 \pmod 8$, and  $|R| \ge 8$.
\item[\rm(iii)]  $p=2$,  $G=E(G)R$ and $E(G)  = T_1 \times \ldots \times T_t, t > 1$ is a minimal normal subgroup 
and if $T = T_1$, then $N_G(T)/C_G(T)$ has a maximal Sylow $2$-subgroup and  $N_G(T)/C_G(T)$ is isomorphic to one of
\[
{\rm PGL}_2(7), \; {\rm M}_{10},  \; {\rm L}_2(q),\; {\rm PGL}_2(q),
\]
where $q >7$ is a Mersenne prime.
\end{enumerate}
\end{prop}

To see the conjecture holds for a given $G$, it is sufficient to find a $g$ that is not a $p$-element with 
$[g,{}_kx]: = g$. 



\section{Nonlinear multiplicative irreducible characters}\label{s:characters}

In this final section, we present an application of Theorem \ref{t:thmC} to the character theory of finite groups.
Let $G$ be a finite group and let $\chi$ be an irreducible complex character of $G$. Motivated by the concept of multiplicative functions in analytic number theory, Guralnick and Moret\'{o}  \cite{GM} call $\chi$ a  multiplicative character if $\chi(xy)=\chi(x)\chi(y)$ for every nontrivial elements $x,y\in G$ with $(o(x),o(y))=1$. Clearly, every linear character of $G$ is multiplicative. To obtain further examples of multiplicative characters, we need the following notation and concepts from character theory.

We write $\Irr(G)$ for the set of all complex irreducible characters of $G$ and let $\chi\in\Irr(G)$.  We say that $\chi$ vanishes at $g\in G$ if $\chi(g)=0$. If $N$ is a normal subgroup of $G$, we say that $\chi$ vanishes off $N$ if $\chi(g)=0$ for every element $g\in G-N.$ If $g\in G$, then we can write $g=g_pg_{p'}=g_{p'}g_p$, where $g_p$ is a $p$-element, and $g_{p'}$ is a $p'$-element of $G$.
Note that if $\chi$ vanishes off a normal $p$-subgroup of $G$, then $\chi$ is multiplicative.


 Below are some examples of groups with a nonlinear multiplicative character.
\begin{ex} Let $G$ be a finite group and let $p$ be a prime.
\begin{enumerate}[\rm (i)]
\item Recall that a finite group $G$ with $|G|>2$ is called a Gagola group if $G$ has an irreducible character $\chi$
that vanishes on all but two conjugacy classes of $G$. The character $\chi$ above is called a Gagola character. In \cite{Gagola}, Gagola shows that every Gagola group with a Gagola character $\chi$ has a unique minimal normal subgroup $N$ which is an elementary abelian $p$-group for some prime $p$ and that $\chi$ vanishes off $N$. Thus, $\chi$ is multiplicative. 

\item If $G$ is a Frobenius group with Frobenius kernel a $p$-group for some prime $p$, then any nonlinear faithful irreducible character of $G$ is multiplicative. 

\item Trivially, every nonlinear irreducible character of a finite $p$-group is multiplicative.

\item Let $K$ be a proper nontrivial normal subgroup of $G$. The pair $(G,K)$ is called a Camina pair if for every element $g\in G-K$, then $g$ is conjugate to every element in the coset $gK$. Equivalently, $G$ is a Camina pair if and only if every irreducible character $\chi$ of $G$ that does not contain $K$ in its kernel vanishes off $K$ (see \cite[Lemma 4.1]{Lewis}). A result of Camina, (see \cite[Theorem 4.4]{Lewis}) states that if $(G,K)$ is a Camina pair, then either $G$ is a Frobenius group with Frobenius kernel $K$ or one of $G/K$ or $K$ is a $p$-group for some prime $p.$ Thus if $(G,K)$ is a Camina group and $K$ is a $p$-group for some prime $p$, then every nonlinear irreducible character of $G$ that does not contain $K$ in its kernel is multiplicative.
\end{enumerate}
\end{ex}

An irreducible character $\chi\in\Irr(G)$ is said to have $p$-defect zero (or $\chi$ lies in a block of $p$-defect $0$) if $\chi(1)_p=|G|_p,$ where $n_p$ denotes the  $p$-part of the integer $n\ge 1$. The following result due to Kn\"{o}rr characterizes $p$-defect zero irreducible characters.

\begin{lem}\label{lem: defect 0}
Let $G$ be a finite group, $p$ be a prime and $\chi\in\Irr(G)$.  Then the following are equivalent.
\begin{enumerate}[\rm (i)]
\item $\chi$ has $p$-defect zero.
\item  $\chi$ vanishes on every element of order $p$ of $G$.
\item $\chi$ vanishes on all $p$-singular element of $G$.
\end{enumerate}
\end{lem}

\begin{proof}
This is part of Corollary 2.1 in \cite{Knorr}
\end{proof}

We also need to the following result which is a special case of Lemma 2.2 in \cite{GMT}.

\begin{lem}\label{lem:product}
Let $G$ be a finite group and let $a,b\in G$. Let $A=a^G$ and $B=b^G.$ If $\chi\in\Irr(G)$ is constant on $AB$, then $\chi(a)\chi(b)=\chi(ab)\chi(1)$.
\end{lem}

We first prove the following.

\begin{thm}\label{th:p-core}
Let $G$ be a  finite group. Suppose that $\chi\in\Irr(G)$ is a nonlinear  multiplicative character. Then 
\begin{enumerate}[\rm (i)]
\item If $a,b\in G$ are nontrivial and $(o(a),o(b))=1$, then $\chi(a)=0$ or $\chi(b)=0$. In particular, $\chi(ab)=0$.
\item There exists a prime $p$ and  an element $w\in G$ of order $p$ such that $\chi(w)\neq 0.$
\item Let $p$ be a prime such that $\chi(w)\neq 0$ for some $w\in g$ of order $p$. Then:
\begin{enumerate}
    \item[\rm(a)] $\chi(g)=0$ if $g\in G$ is not a $p$-element.
\item[\rm(b)] $\chi$ vanishes off $O_p(G)$.
\item[\rm(c)] $|G|/\chi(1)$ is a power of $p$, $\chi=\lambda^G$ for some $\lambda\in\Irr(P)$,  and $F^*(G)= O_p(G)$, where $P\in\Syl_p(G)$ and $  F^*(G)$ is the generalized Fitting subgroup of $G$.
\end{enumerate}

\end{enumerate}
\end{thm}

\begin{proof}

Recall that $\chi(xy)=\chi(x)\chi(y)$ for all $1\neq x,y\in G$ with $(o(x),o(y))=1.$

(i) Let $a,b\in G$ be nontrivial with $(o(a),o(b))=1$. Let $A=a^G$ and $B=b^G.$ Then for any $c\in AB$, $c=a^ub^v$ for some $u,v\in G.$ As $o(a^u)=o(a)$ and $o(b^v)=o(b)$ and $\chi\in\Irr(G)$ is a class function on $G$, we see that $$\chi(c)=\chi(a^ub^v)=\chi(a^u)\chi(b^v)=\chi(a)\chi(b).$$  Thus $\chi$ is constant on $AB$ and so by Lemma \ref{lem:product}, $\chi(ab)\chi(1)=\chi(a)\chi(b)$ which implies that $\chi(ab)\chi(1)=\chi(ab)$. As $\chi(1)>1$ ($\chi$ is nonlinear),  $\chi(ab)=0$, hence $\chi(a)\chi(b)=\chi(ab)=0$ and (1) follows.

(ii) Assume by contradiction that $\chi$ vanishes on every element of prime order in $G$. Then by Lemma \ref{lem: defect 0}, $\chi$ has $r$-defect zero, that is, $\chi(1)_r=|G|_r$,  for every prime divisor $r$ of $|G|$, the order of $G$. However, this would imply that $\chi(1)=|G|$, which is impossible as $\chi(1)^2<|G|$. Therefore, $\chi$ does not vanish on some element, say $w$, of order $p$, for some prime $p.$ 

(iii)(a) Suppose that $1\neq g\in G$ is not a $p$-element. Then $g$ is either a $p'$-element or $g$ is $p$-singular but not a $p$-element. Assume first that $g$ is a $p'$-element. Then by part (i), we have $\chi(wg)=\chi(w)\chi(g)=0$. As $\chi(w)\neq 0$, $\chi(g)=0.$ Next, assume that $g$ is $p$-singular but not a $p$-element. Then $g=uv$, where both $u=g_p,v=g_{p'}$ are nontrivial elements of $G$ with $(o(u),o(v))=1$. Part (i) now implies that $\chi(g)=0.$ Thus $\chi(g)=0$ if $g\in G$ is not a $p$-element.

(iii)(b) Let $1\neq x\in G$ with $\chi(x)\neq 0$. It follows from part (iii)(a) that $x$ is a nontrivial $p$-element. Let $y\in G$ be a nontrivial $p$-element of $G$. If $xy=s$ is a nontrivial $p'$-element, then $x=sy^{-1}$ with both $s,y^{-1}$ nontrivial and $(o(s),o(y^{-1}))=1$ so that by part (i), we have $\chi(x)=\chi(sy^{-1})=0$, which is a contradiction. Therefore, $xy$ is $1$ or $p$-singular for every $p$-element $y\in G.$ Now by Theorem \ref{t:thmC}, $x\in  O_p(G)$ and the results follows.

(iii)(c) For each prime divisor $r$ of $|G|$ with $r\neq p$, $\chi$ vanishes on every element of order $r$ and so by Lemma \ref{lem: defect 0}, $\chi(1)_r=|G|_r$ and thus $\chi(1)_{p'}=|G|_{p'}$ or equivalently $|G|/\chi(1)$ is a  power of $p.$ The remaining claims follow from  Lemma 1 and Theorem B in \cite{RS}.
\end{proof}

We now prove the main result of this section, answering a question raised in \cite{GM}.

\begin{thm}\label{th:app} Let $G$ be a finite group. Let $\chi$ be a nonlinear  irreducible character of $G$. Then $\chi$ is multiplicative if and only if there is a prime $p$ such that $\chi$ vanishes off  $  O_p(G)$.
\end{thm}

\begin{proof}Assume first that $\chi\in \Irr(G)$ is nonlinear multiplicative. By Theorem \ref{th:p-core}(iii)(b), $\chi$ vanishes off $  O_p(G)$. 
Conversely, assume that $\chi$ vanishes off $  O_p(G)$. We claim that $\chi$ is multiplicative. Let $x,y\in G$ be nontrivial elements with $(o(x),o(y))=1.$ Let $z=xy.$ Since $x$ and $y$ have coprime orders, we may assume that $p\nmid o(x)$. It follows that $x\not\in  O_p(G)$ and thus $\chi(x)=0$. Now, if $\chi(z)=0$, then $\chi(xy)=\chi(z)=0=\chi(x)\chi(y)$. Thus we may assume that $\chi(z)\neq 0$, hence $z\in  O_p(G).$ In the quotient group $\overline{G}=G/  O_p(G)$, we see that $\overline{x}\overline{y}=\overline{1}$ and hence $\overline{y}=\overline{x}^{-1}$. So $o(\overline{y})=o(\overline{x})>1$, which is impossible.
\end{proof}


\end{document}